\documentclass{amsart}

\usepackage{latexsym,ifthen,amssymb}
\usepackage{latexsym,ifthen,amssymb}
\usepackage[square,numbers]{natbib}
\usepackage[toc,page,title,titletoc,header]{appendix}
\usepackage{multirow}
 \usepackage{longtable}
 \usepackage{rotating}
 \usepackage{graphicx, amsmath, amsthm, amssymb,float,setspace,color, multirow}
\usepackage{enumitem}
\usepackage{graphicx}
\usepackage{bm}
\usepackage{subfigure}
\usepackage{float}
\usepackage{rotating}
\usepackage{epstopdf}
\usepackage{latexsym,ifthen,amssymb}
\usepackage{natbib}
\usepackage{algorithm}
\usepackage{algorithmicx}
\usepackage{algpseudocode}
\usepackage[toc,page,title,titletoc,header]{appendix}
\usepackage{multirow}
 \usepackage{longtable}
 \usepackage{rotating}
 \usepackage{graphicx, amsmath, amsthm, amssymb,float,setspace,color, multirow}
\usepackage{tikz}
\usepackage{hyperref}
\usepackage{enumitem}
\usepackage{color, soul}

\usetikzlibrary{decorations.markings}

\usepackage{url}
\usepackage{xcolor}	
\usepackage{soul}

\newtheorem{theorem}{Theorem}[section]
\newtheorem{lemma}[theorem]{Lemma}

\newtheorem{corollary}[theorem]{Corollary}
\newtheorem{proposition}[theorem]{Proposition}

\theoremstyle{definition}
\newtheorem{definition}[theorem]{Definition}

\theoremstyle{remark}

\newtheorem {question}[theorem]{Question}
\numberwithin{equation}{section}

\newtheoremstyle{noparens}%
  {}{}%
{}{}%
{\bfseries}{.}%
{ }%
{\thmname{#1}\thmnumber{ #2}\thmnote{ #3}}
\theoremstyle{noparens}
\newtheorem*{question*}{Question}
\newtheorem*{theorem*}{Theorem}

\title{The reverse mathematics of the thin set and Erd\H{o}s-Moser theorems}

\author{Lu Liu}
\address{School of Mathematics and Statistics, HNP-LAMA\\
Central South University\\
ChangSha 410083\\
People’s Republic of China}
\email{g.jiayi.liu@gmail.com}

\author{Ludovic Patey}
\address{CNRS, Institut Camille Jordan\\
Universit\'{e} Claude Bernard Lyon 1\\
43 boulevard du 11 novembre 1918, F-69622 Villeurbanne Cedex, France}
\email{ludovic.patey@computability.fr}

\thanks{}

\def\t{\tilde}
\def\propertyL{2-hyperimmunity}
\def\hyper{2-hyperimmune}
\def\msf{\mathsf}
\def\mcal{\mathcal}

\def\v{\vec}
\newcommand{\imp}{\rightarrow}

\newcommand{\Nb}{\mathbb{N}}

\newcommand{\Psf}{\mathsf{P}}
\newcommand{\Qsf}{\mathsf{Q}}

\newcommand{\Bcal}{\mathcal{B}}

\newcommand{\Fcal}{\mathcal{F}}

\newcommand{\Hcal}{\mathcal{H}}
\newcommand{\Ical}{\mathcal{I}}

\newcommand{\Mcal}{\mathcal{M}}
\newcommand{\Pcal}{\mathcal{P}}

\newcommand{\Rcal}{\mathcal{R}}

\renewcommand{\setminus}{\smallsetminus}

\newcommand{\p}[1]{\left( #1 \right)}

\newcommand{\card}[1]{\left| #1 \right|}

\newcommand{\cond}[1]{\left\{\begin{array}{ll} #1 \end{array}\right.}

\newcommand{\s}[1]{\ensuremath{\sf{#1}}}

\DeclareMathOperator{\rca}{\s{RCA}_0}
\DeclareMathOperator{\aca}{\s{ACA}}
\DeclareMathOperator{\wkl}{\s{WKL}}

\DeclareMathOperator{\dnr}{\s{DNR}}

\DeclareMathOperator{\rt}{\s{RT}}

\DeclareMathOperator{\ads}{\s{ADS}}
\DeclareMathOperator{\sads}{\s{SADS}}

\DeclareMathOperator{\coh}{\s{COH}}

\DeclareMathOperator{\ts}{\s{TS}}
\DeclareMathOperator{\sts}{\s{STS}}

\DeclareMathOperator{\fs}{\s{FS}}

\DeclareMathOperator{\emo}{\s{EM}}
\DeclareMathOperator{\semo}{\s{SEM}}

\def\iatfs{array}
\def\r{\rightarrow}
\def\iatfshyp{4-hyperimmune}
\def\propiatfshyp{4-hyperimmunity}
\def\compatible{compatible}
\def\diagonalagainst{diagonalizes against}
\def\h{\hat}
\def\fssufficient{$\fs$-sufficient}
\def\tssufficient{$\ts$-sufficient}

\begin{document}

\begin{abstract}
The thin set theorem for $n$-tuples and $k$ colors ($\ts^n_k$)
states that every $k$-coloring of $[\Nb]^n$ admits an infinite set of integers $H$
such that $[H]^n$ avoids at least one color.
In this paper, we study the combinatorial weakness of the thin set theorem
in reverse mathematics by proving neither $\ts^n_k$, nor the free set theorem ($\fs^n$) imply
the Erd\H{o}s-Moser theorem ($\emo$) whenever $k$ is sufficiently large (answering a question of Patey and giving a partial result towards a question of Cholak Giusto, Hirst and Jockusch).
Given a problem $\msf{P}$, a computable instance of $\msf{P}$
is universal iff its solution computes a solution of any other computable
$\msf{P}$-instance. It has been established that most of Ramsey-type problems
do not have a universal instance, but the case of Erd\H{o}s-Moser theorem
remained open so far.
We prove that Erd\H{o}s-Moser theorem does not admit a universal instance (answering a question of Patey).

\end{abstract}

\maketitle

\section{Introduction}

In this paper, we study the computability-theoretic strength of Ramsey's theorem
when we weaken the notion of homogeneous set by allowing more colors. The resulting weakenings are known as thin set theorems. In particular,
we show that some thin set theorems are sufficiently weak not to imply the Erd\H{o}s-Moser theorem in reverse mathematics. 

Reverse mathematics is a foundational program which seeks to determine the optimal axioms
to prove ``ordinary'' theorems~\cite{Friedman1974Some}. It uses the framework of subsystems of second-order arithmetic,
with a base theory called $\rca$, informally capturing ``computable mathematics''.
When the first-order part consists of the standard integers, the models of $\rca$
are fully specified by their second-order parts, which are precisely the \emph{Turing ideals}.
A Turing ideal $\Ical$ is a collection of sets which is closed downward under the Turing
reduction $(\forall X \in \Ical)(\forall Y \leq_T X) Y \in \Ical$ and closed under the
effective join $(\forall X, Y \in \Ical) X \oplus Y \in \Ical$. We shall only consider
such models, called $\omega$-models.

The early study of reverse mathematics has seen the emergence
of 4 subsystems of second-order arithmetic linearly ordered by the provability relation,
such that most of the ordinary theorems are either provable in $\rca$, or equivalent in $\rca$ to one of them.
These subsystems, together with $\rca$, form the ``Big Five''~\cite{Montalban2011Open}.
Among them, let us mention $\aca$, standing for arithmetical comprehension axiom,
and $\wkl$, standing for weak K\"onig's lemma, which states that every infinite binary
tree has an infinite path. See Simpson~\cite{Simpson2009Subsystems}
for a good introduction to reverse mathematics. Among the theorems studied in reverse mathematics,
Ramsey's theorem plays an important role, since Ramsey's theorem for pairs is historically the first example
of statement which does not satisfy this empirical observation.

\begin{definition}[Ramsey's theorem]
A subset~$H$ of~$\omega$ is~\emph{homogeneous} for a coloring~$f : [\omega]^n \to k$ (or \emph{$f$-homogeneous})
if each $n$-tuple over~$H$ is given the same color by~$f$.
$\rt^n_k$ is the statement ``Every coloring $f : [\omega]^n \to k$ has an infinite $f$-homogeneous set''.
\end{definition}

$\rt^1_k$ is nothing but the infinite pigeonhole principle which is provable in $\rca$.
Jockusch~\cite{Jockusch1972Ramseys} (see Simpson~\cite{Simpson2009Subsystems})
proved that $\rt^n_k$ is equivalent to $\aca$ over $\rca$ whenever~$n \geq 3$,
and that $\wkl$ does not imply $\rt^2_2$ over $\rca$.
The computability-theoretic strength of Ramsey's theorem for pairs was unknown for a long time,
until Seetapun~\cite{Seetapun1995strength} proved that $\rt^2_2$ does not imply $\aca$ over~$\rca$,
and that the first author~\cite{Liu2012RT22} proved that $\rt^2_2$ does not imply $\wkl$ over~$\rca$, thereby
showing that $\rt^2_2$ is not even linearly ordered with the Big Five.

This analysis of Ramsey's theorem naturally started new research axis,
among which the search for weakenings of Ramsey's theorem for arbitrary $n$-tuples
which would not imply $\aca$. Ramsey's theorem can be seen as a problem,
whose \emph{instances} are $k$-colorings of $[\omega]^n$, and whose \emph{solutions} are infinite homogeneous sets.
This problem has two explicit parameters, namely, the size $n$ of the $n$-tuples, and the number $k$ of colors of the coloring.
There is one implicit parameter which is the number of colors~$\ell$ allowed in the solution.
In the case of a homogeneous set, $\ell = 1$. In this paper, we give a partial answer to the following question.

\begin{quote}
$\bullet$\ How does the number of colors allowed in a solution impact
the computability-theoretic strength of Ramsey's theorem?
\end{quote}

We are in particular interested in the case where $\ell = k-1$. This yields the notion of thin set.

\begin{definition}[Thin set theorem]
Given a coloring~$f : [\omega]^n \to k$ (resp.\ $f : [\omega]^n \to \omega$), a set~$H$
is \emph{thin} for~$f$ (or \emph{$f$-thin}) if~$|f([H]^n)| \leq k-1$ (resp.\ $f([H]^n) \neq \omega$).
For every~$n \geq 1$ and~$k \geq 2$, $\ts^n_k$ is the statement
``Every coloring $f : [\omega]^n \to k$ has an infinite $f$-thin set''
and $\ts^n_\omega$ is the statement ``Every coloring $f : [\omega]^n \to \omega$ has an infinite $f$-thin set''.
\end{definition}

In particular, $\ts^n_2$ is Ramsey's theorem for $n$-tuples and 2 colors.
The thin set theorem $\ts^n_\omega$ was introduced in reverse mathematics
by Friedman~\cite{FriedmanFom53free} and studied by Cholak, Giusto,
Hirst and Jockusch~\cite{Cholak2001Free}. Wang~\cite{Wang2014Some} proved the surprising
result that for every $n \geq 1$ and every sufficiently large $k$ (with an explicit upper bound on $k$),
$\ts^n_k$ does not imply $\aca$, thereby showing that allowing more colors in the solutions
yields a strict weakening of Ramsey's theorem, which is already reflected in reverse mathematics.
The second author refined Wang's analysis by proving that
for every $n,m,\ell\in\omega$ with $m>1$ and every sufficiently large $k$,
$\ts^n_k$ implies neither $\wkl$~\cite{PateyCombinatorial},
nor $\ts^m_\ell$~\cite{Patey2016weakness}. In particular,
the statement $(\forall n)(\exists k)\ts^n_k$ does not imply $\rt^2_2$ over $\rca$. In this paper,
we partially answer the following sub-question.

\begin{quote}
$\bullet$\ What consequences of Ramsey's theorem for pairs are already consequences of the various thin set theorems
in reverse mathematics?
\end{quote}

The thin set theorem for pairs with $\omega$ colors ($\ts^2_\omega$) seems combinatorially very weak,
however, it has a diagonalization power similar to Ramsey's theorem for pairs. For example,
there is a computable instance of $\ts^2_\omega$ with no $\Sigma^0_2$ solution~\cite{Cholak2001Free},
and given any low${}_2$ set $X$, there is a computable instance of $\ts^2_\omega$ with no $X$-computable solution~\cite{patey2015degrees}.
One can strengthen the thin set theorem by asking, given a coloring $f : [\omega]^n \to \omega$,
for an infinite set $H$ such that $H \setminus \{x\}$ is $f$-thin for every $x \in H$.
This yields the notion of free set.

\begin{definition}[Free set theorem]
Given a coloring $f : [\omega]^n \to \omega$, a set~$H$
is \emph{free} for~$f$ (or \emph{$f$-free}) if for every~$\sigma \in [H]^n$,
$f(\sigma) \in H \imp f(\sigma) \in \sigma$.
For every~$n \geq 1$, $\fs^n$ is the statement
``Every coloring $f : [\omega]^n \to \omega$ has an infinite $f$-free set''.
\end{definition}

The free set theorem is usually studied together with the thin set theorem,
as they are combinatorially very close. The standard proof of $\fs^n$ involves
the statement $(\exists k)\ts^n_k$ in a way that propagates most of the computability-theoretic properties of
the thin set theorem to the free set theorem. This is why any known proof that $(\exists k)\ts^n_k$
does not imply another statement $\Psf$ over $\rca$ empirically yields a proof that $\fs^n$
does not imply~$\Psf$~\cite{Cholak2001Free,Wang2014Some,Patey2016weakness,PateyCombinatorial}.
This will again be the case in our paper.

\subsection{Main results}
Ramsey's theorem for pairs
admits various decompositions into conjunctions of strictly weaker
 statements. Among
them, the decomposition into the Erd\H{o}s-Moser theorem and the Ascending Descending sequence principle
is particularly interesting for various technical reasons.

\begin{definition}[Erd\H{o}s-Moser theorem]
A tournament $T$ is an irreflexive binary relation such that for all $x,y \in \omega$ with $x \not= y$, exactly one of $T(x,y)$ or $T(y,x)$ holds. A set~$H$ is \emph{$T$-transitive} if the relation $T$ over $H$ is transitive in the usual sense.
$\emo$ is the statement ``Every infinite tournament $T$ has an infinite transitive subtournament.''
\end{definition}

\begin{definition}[Ascending descending sequence]
Given a linear order~$(L, <_L)$, an \emph{ascending} (\emph{descending}) sequence
is a set~$S$ such that for every~$x <_\Nb y \in S$, $x <_L y$ ($x >_L y$).
$\ads$ is the statement ``Every infinite linear order admits an infinite ascending or descending sequence''.
\end{definition}

Bovykin and Weiermann~\cite{Bovykin2005strength} proved Ramsey's theorem for pairs as follows:
Given a coloring~$f : [\Nb]^2 \to 2$, we can see~$f$ as a tournament~$T$
such that whenever~$x <_\Nb y$, $T(x,y)$ holds if and only if~$f(x,y) = 1$.
Any $T$-transitive set $H$ can be seen as a linear order $(H, \prec)$
such that for every~$x <_\Nb y$: ~$x \prec y$ if and only if~$f(x,y) = 1$.
Any infinite ascending or descending sequence is $f$-homogeneous.
It is therefore natural to study the ascending descending sequence principle
together with the Erd\H{o}s-Moser theorem.
Lerman, Solomon and Towsner~\cite{Lerman2013Separating} proved that $\emo$ does not imply $\ads$ over~$\rca$,
while Hirschfeldt and Shore~\cite{Hirschfeldt2007Combinatorial} proved that $\ads$ does not imply $\emo$.
The second author asked~\cite{Patey2016weakness} whether any of $\fs^2$, $\ts^2_\omega$, or $\ts^2_3$
implies $\emo$ over~$\rca$. We answer this question negatively, even for
stable restrictions of the Erd\H{o}s-Moser   theorem.

\begin{theorem}\label{thm:many-and-ts24-not-sem-and-many}
Over $\rca$, $\wkl + \coh + \ts^2_4 + (\forall n)(\exists k)\ts^n_k + (\forall n)\fs^n$
implies none of $\semo$, $\sts^2_3$ and $\sads$.
\end{theorem}

In  Theorem \ref{thm:many-and-ts24-not-sem-and-many},
 $\sts^2_k$, $\semo$ and $\sads$ denote the restriction
of $\ts^2_k$, $\emo$ and $\ads$ to stable colorings, respectively.
A coloring $f : [\omega]^2 \to k$ is \emph{stable} if for every $x \in \omega$,
$\lim_s f(x, s)$ exists. In the case of $\sads$, this yields the statement ``Every linear order
of type $\omega+\omega^{*}$ has an infinite ascending or descending sequence.''
$\coh$ is the statement ``For every sequence of sets $\vec R=(R_0, R_1, \dots)$,
there is an infinite set $C$ such that $C \subseteq^{*} R_i$ or $C \subseteq^{*} \overline{R}_i$
for every $i \in \omega$.'' Such set $C$ is called \emph{$\v R$-cohesive}

The separation result, Theorem \ref{thm:many-and-ts24-not-sem-and-many},
  shows that although the thin set theorem
shares many lower bounds with Ramsey's theorem, allowing more colors in the solutions
yields a statement with strictly weaker computability-theoretic properties.
Cholak, Giusto,
Hirst and Jockusch~\cite{Cholak2001Free}
Question 7.4 asked whether $\msf{FS}^2+\msf{CAC}$ implies $\msf{RT}_2^2$.
Bovykin and Weiermann~\cite{Bovykin2005strength} proved
that $\msf{RT}_2^2$ is equivalent to $\msf{EM}+\msf{ADS}$.
It's well known $\msf{CAC}$ implies $\msf{ADS}$.
Thus Theorem \ref{thm:many-and-ts24-not-sem-and-many} ($\msf{FS}^2$ does not imply
$\msf{EM}$) is
a partial result towards a negative answer to Cholak, Giusto,
Hirst and Jockusch's question.

For a problem $\msf{P}$, a
computable $\msf{P}$-instance $I^*$ is
 \emph{universal} iff for every computable $\msf{P}$-instance
 $I$, every solution of $I^*$ computes a solution of $I$.
 The most well known example of a universal instance is
 for $\msf{WKL}$, there is a computable tree $T^*\subseteq 2^{<\omega}$ so that every infinite path through
  $T^*$ is of PA degree. Thus every infinite path through $T^*$
 computes an infinite path of a given computable infinite tree $ T$.
 Patey \cite{patey2015degrees} systematically studied which Ramsey type problem admits universal instance.
Many Ramsey type problems do not admit a universal instance.
Often, if a problem admits a universal instance, it is relatively easy to construct
one.  The coding is not hard when it exists.
 For several problems, the question remains.
The second author asked in \cite{Patey2015Open} that whether $\msf{EM}$
admits a universal instance.
We answer this question negatively.
\begin{theorem}
\label{uemth1}
$\msf{EM}$ does not have universal instance.
\end{theorem}
The first author asked a similar question with respect to an arbitrary instance
of $\msf{RT}_2^1$.
Clearly, when an instance is universal, it encodes information about
every other computable instance.
For a problem $\msf{P}$, we consider the mass problem generated by instances of $\msf{P}$.
For $\msf{P}$-instances $I,\h I$ (not necessarily computable), we say $I$ \emph{encodes} $\h I$ iff
every solution of $I$ computes a solution of $\h I$. That is,
the set of solutions of $\h I$ is Muchnick reducible to that of $I$.
Liu \cite{Liu2015Cone} asked whether there is a $\msf{RT}_2^1$ instance $X$
 that is maximal
(in the lattice of the encoding relation) in the sense that
for every $\msf{RT}_2^1$ instance $Y$, if $Y$ encodes $X$, then
$X$ encodes $Y$.

\subsection{Organization}
The paper is divided into two main sections, corresponding to the proofs
of Theorem~\ref{thm:many-and-ts24-not-sem-and-many} and Theorem~\ref{uemth1},
respectively. In section~\ref{sect:fs-ts-sem}, we introduce a framework
for preservation of \propertyL, and develop its basic properties in subsection~\ref{subsect:double-immunity-related}.
We then prove in subsection~\ref{subsect:ts24-preserves} that $\ts^2_4$ preserves \propertyL, and then generalize the proof
to $(\forall n)(\exists k)\ts^n_k$ in subsection~\ref{subsect:tsn-preserves}. We prove that $\fs^2$
preserves \propertyL\ in subsection~\ref{subsect:fs2-preserves}, and again generalize it to $(\forall n)\fs^n$
in subsection~\ref{subsect:fsn-preserves}.
In section \ref{uemsec1}, we prove Theorem \ref{uemth1}.

\subsection{Notation}

Given two sets $A$ and $B$, we write $A < B$ iff $x<y$ for all $x\in A,y\in B$ ($A < \emptyset$ and $\emptyset < B$
both hold); we write $A>y$ iff $x>y$ for all $x\in A$.
We use $|A|$ to  denote the cardinal of the set $A$.
Denote by $[A]^n$ the collection of $n$-element subsets of $A$;
$[A]^{<\omega}$ the collection of finite subsets of $A$.
Usually, we use $F,E,D$ and sometimes $\sigma,\tau$ to denote finite sets of integers;
$X,Y,Z$ to denote infinite sets of integers;
$A,B,H,G$ to denote sets of integers.


A \emph{Mathias condition} is a pair $(F, X)$
where $F$ is a finite set, $X$ is an infinite set.
 A condition $(E, Y)$ \emph{extends } $(F, X)$ (written $(E, Y) \leq (F,X)$)
if $E\supseteq F, E\setminus F>F, E\setminus F\subseteq X, Y\subseteq X$. Note that we do not require $F < X$ for a condition $(F, X)$, but continuity is ensured by the extension relation.
A set $G$ \emph{satisfies} a Mathias condition $(F, X)$
if $F \subseteq G$, $G \setminus F \subseteq X, G\setminus F>F$.
A Mathias condition $c$ is seen as the collection $\{G:G\text{ satisfies }(F,X)\}$;
we define $c\leq d$ iff the collection of $c$ is a sub collection of   $d$.
We adopt the convention that if $\Phi^F(n)\downarrow$ and $G>F$,
then $\Phi^{F\cup G}(n)\downarrow =\Phi^F(n)$.

\section{Free and thin sets which are not transitive}\label{sect:fs-ts-sem}

This section is devoted to the proof of Theorem~\ref{thm:many-and-ts24-not-sem-and-many} that we recall now.

\begin{theorem*}[\ref{thm:many-and-ts24-not-sem-and-many}]
Over $\rca$, $\wkl + \coh + \ts^2_4 + (\forall n)(\exists k)\ts^n_k + (\forall n)\fs^n$
implies none of $\semo$, $\sts^2_3$ and $\sads$.
\end{theorem*}

For this, we are going to prove that $\wkl$, $\coh$, $\ts^2_4$ and $\fs^n$
preserve a computability-theoretic notion that we call \propertyL,
while none of $\semo$, $\sts^2_3$ and $\sads$ does (see Lemma \ref{lem:transformtopreservation}
for why this is enough).
We introduce the
concept of 2-hyperimmunity preservation next.  The required results about
2-hyperimmunity preservation are proved in the remaining sections of this section
(see Figure \ref{fig:framework0} for where they are proved).
\begin{definition}\
\begin{itemize}
	\item[(1)] A \emph{bifamily} is a collection $\Hcal$
of ordered pairs of finite  sets which is closed downward under the product subset relation,
that is, such that if $(C, D) \in \Hcal$ and $E \subseteq C$ and $F \subseteq D$,
then $(E, F) \in \Hcal$.

	\item[(2)] A \emph{biarray} is a collection of finite  sets
$(\vec{E},\vec{F}) = \langle E_n, F_{n,m} : n, m \in \omega \rangle$ such that
$E_n >n$ and $F_{n,m} > m$ for every $n, m \in \omega$.
A biarray $(\vec{E},\vec{F})$ \emph{meets} a bifamily $\Hcal$ if there
is some $n, m \in \omega$ such that $(E_n, F_{n,m}) \in \Hcal$.

	\item[(3)] A bifamily $\Hcal$ is $C$-\emph{\hyper}\
if every $C$-computable biarray meets $\Hcal$.
\end{itemize}
\end{definition}

We shall relate the notion of \propertyL\ and various notions
of immunity in section~\ref{subsect:double-immunity-related}.
For notational convenience,
in this section we regard each Turing machine $\Phi$ as computing
a biarray. We will therefore assume that whenever $\Phi(n;1)$ converges,
then it will output (the canonical index of) a finite  set $E_n > n$ and
 whenever $\Phi(n,m;2)$
converges, then it will output a finite  set $F_{n,m} > m$.

\begin{definition}Fix a problem $\Psf$.
\begin{itemize}
	\item[(1)] $\msf{P}$ \emph{preserves} \propertyL\ if
for every  bifamily $\Hcal$ that is $C$-\hyper,
every $C$-computable $\msf{P}$-instance
admit a solution $G$ such that
$\Hcal$ is $C\oplus G$-\hyper.

	\item[(2)] $\msf{P}$ \emph{strongly preserves} \propertyL\ if
for every bifamily $\Hcal$ that is $C$-\hyper,
every $\msf{P}$-instance
admit a solution $G$ such that
$\Hcal$ is $C\oplus G$-\hyper.
\end{itemize}
\end{definition}

The following lemma is a particular case of Lemma~3.4.2 in~\cite{Patey2016reverse}.
We reprove it for the sake of  completeness.

\begin{lemma}\label{lem:transformtopreservation}
If some problems $\Psf_1, \Psf_2, \dots$ preserve \propertyL\
while another problem $\Qsf$ does not, then the conjunction $\bigwedge_i \Psf_i$
does not imply $\Qsf$ over $\rca$.
\end{lemma}
\begin{proof}
Since $\Qsf$ does not preserve \propertyL, there is some set $C$,
a bifamily $\Hcal$ that is $C$-\hyper, and a $\Qsf$-instance $B$
such that for every solution $G$, $\Hcal$ is not $C \oplus G$-\hyper.
Since each $\Psf_i$ preserve \propertyL, we can define an infinite
sequence of sets $C = Z_0 \leq_T Z_1 \leq_T \dots$ such that
\begin{itemize}
	\item[(i)] $\Hcal$ is $Z_n$-\hyper\ for every $n$
	\item[(ii)] For every $i, n \in \omega$, every $Z_n$-computable
		$\Psf_i$-instance has a $Z_m$-computable solution for some $m$
\end{itemize}
Consider the $\omega$-structure $\Mcal = \{ X : (\exists n) X \leq_T Z_n \}$.
By construction, $B \in \Mcal$,
and by (i), $\Hcal$ is $C \oplus G$-\hyper\ for every $G \in \Mcal$.
It follows that the $\Qsf$-instance $B \in \Mcal$ has no solution in $\Mcal$, so $\Mcal \not \models \Qsf$.
By (ii), $\Mcal \models \Psf_i$ for every $i \in \omega$. This completes the proof.
\end{proof}

\begin{figure}[htbp]
\begin{center}
\begin{tikzpicture}[x=2.5cm, y=2cm, widget/.style={rectangle, draw,align=center},decoration={
    markings,
    mark=at position 0.5 with {\arrow{>}}}]
	
	\node[widget] (A) at (0, -1) {Corollary~\ref{cor:wkl-double-immunity}\\ $\wkl$ preserves};
	\node[widget] (B) at (0, 1) {Corollary~\ref{cor:coh-preserves-double-immunity}\\ $\coh$ preserves};
	\node[widget] (D) at (-2, 1) {Lemma~\ref{lem:ts-n-to-np1-strongly-double-immune}\\ $\ts^{n}_{d_{n-1}+1}$ preserves};
	\node[widget] (G) at (2, 1) {Lemma~\ref{lem:fs-n-to-np1-strongly-double-immune}\\ $\fs^n$ preserves};
	\node[widget] (C) at (-1, 0) {Induction hypothesis\\ $\ts^s_{d_s+1}$ strongly preserves\\ for $s< n $};
	\node[widget] (F) at (1, 0) {Induction hypothesis\\ $\fs^s$ strongly preserves\\ for $s<n$};
	\node[widget] (E) at (-2, -1) {Theorem~\ref{thm:ts-n-strongly-double-immune-assuming}\\ $\ts^{n}_{d_n+1}$ strongly preserves};
	\node[widget] (H) at (2, -1) {Theorem~\ref{thm:strong-preservation-fs-left-trapped}\\ $\fs^n$ for left trapped functions\\
	 strongly preserves};
	\node[widget] (I) at (2, -2) {Lemma~\ref{lem:fs-n-to-np1-strongly-double-immune}\\ $\fs^n$ strongly preserves};

	\draw[postaction={decorate},thick] (B) -- (D);
	\draw[postaction={decorate},thick] (B) -- (G);
	\draw[postaction={decorate},thick] (C) -- (D);
	\draw[postaction={decorate},thick] (H) -- (I);
	\draw[postaction={decorate},thick] (D) -- (E);
	\draw[postaction={decorate},thick] (G) -- (H);
	\draw[postaction={decorate},thick] (A) -- (E);
	\draw[postaction={decorate},thick] (A) -- (H);
	\draw[postaction={decorate},thick] (F) -- (G);
	\draw[postaction={decorate},thick] (F) -- (H);
	\draw[postaction={decorate},thick] (C) -- (E);
	\draw[postaction={decorate},thick] (C) -- (H);
\end{tikzpicture}

\caption{\label{fig:framework0} Diagram of dependencies between the proofs of preservation of \propertyL.
An arrow from P to Q means that Q depends on P.}
\end{center}

\end{figure}
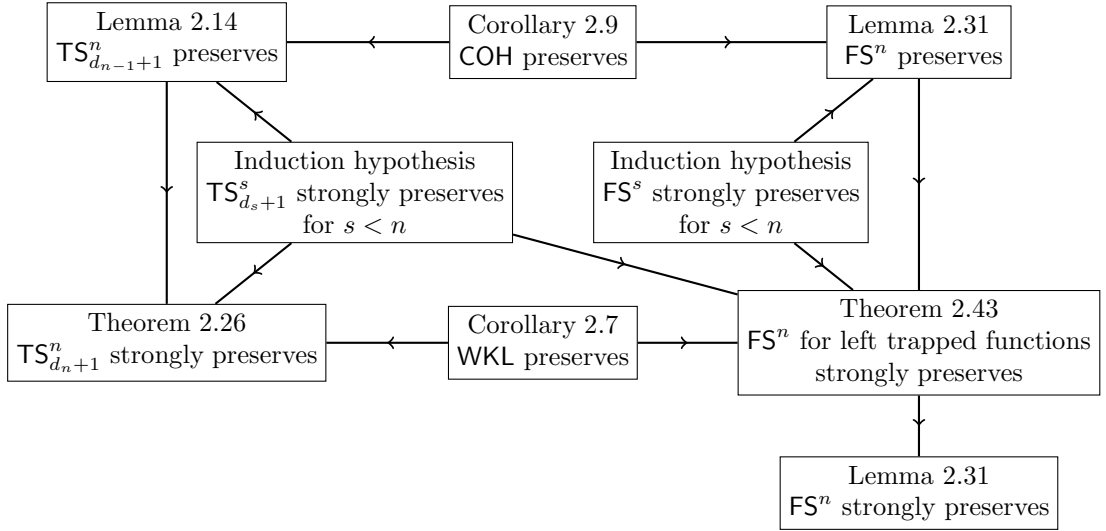

\subsection{Relation with immunity notions}\label{subsect:double-immunity-related}

Given a pair of infinite sets $A, B \subseteq \Nb$,
we let $\Hcal(A,B)$ be the bifamily of all finite  pairs $(E, F)$
such that $E \subseteq \overline{A}$ and $F \subseteq \overline{B}$.
Recall that an infinite set $H$ is $C$-\emph{hyperimmune} if
for every $C$-computable array\footnote{By a computable array
of finite sets $V_0,V_1,\cdots$, we mean a computable function $\alpha:\omega\rightarrow \omega$
such that $\alpha(n)$ is the canonical index of $V_n$.} of finite  sets $V_0, V_1, \dots$
such that $V_n > n$, there is some $n$ such that $V_n \subseteq \overline{H}$.

\begin{lemma}\label{lem:hyper-bifamily}
Two sets $A$ and $B$ are $C$-hyperimmune if and only if $\Hcal(A, B)$ is $C$-\hyper.
\end{lemma}
\begin{proof}
For convenience, we assume $C=\emptyset$
since the result relativizes.
Assume that $A$ and $B$ are hyperimmune,
and fix a computable biarray $(\vec{E}, \vec{F})$.
By hyperimmunity of $A$ applied to $\vec{E}$, there is some $n$ such that $E_n \subseteq \overline{A}$.
By hyperimmunity of $B$ applied to $F_{n,0}, F_{n,1}, \dots$, there is some $m$ such that
$F_{n,m} \subseteq \overline{B}$. It follows that the biarray $(\vec{E}, \vec{F})$ meets $\Hcal(A,B)$.

Assume now that either $A$ or $B$ is not hyperimmune.
Suppose first that $A$ is not hyperimmune. Let $E_0, E_1, \dots$ be a computable array of finite  sets
such that $E_n > n$ and $E_n \cap A \neq \emptyset$ for every $n$.
Then letting $F_{n,m} = \{m+1\}$ for every $m$, the computable biarray $(\vec{E}, \vec{F})$
does not meet $\Hcal(A, B)$.
Suppose now that $B$ is not hyperimmune.
Let $D_0, D_1, \dots$ be a computable array of finite sets
such that $D_n > n$ and $D_n \cap B \neq \emptyset$ for every $n$.
Then letting $E_n = \{n+1\}$ and $F_{n,m} = D_m$ for every $m, n \in \omega$,
the computable biarray $(\vec{E}, \vec{F})$
does not meet $\Hcal(A, B)$. In both cases, $\Hcal(A,B)$ is not  \hyper.
\end{proof}

It follows that if a problem $\Psf$ preserves \propertyL,
then it also preserves 2 hyperimmunities, in the sense of
Definition 6.2.2 in~\cite{Patey2016reverse}. Since $\sads$
is known not to preserve 2 hyperimmunities (see Corollary~10.3.5 in~\cite{Patey2016reverse}),
we can immediatly conclude that $\sads$ does not preserve \propertyL.
We will nevertheless recall the argument.

\begin{corollary}
$\sads$ does not preserve \propertyL.
\end{corollary}
\begin{proof}
Fix any stable linear order of order type $\omega + \omega^{*}$
with no computable infinite ascending or descending subsequence.
Let $A$ and $B$ be the $\omega$ and $\omega^{*}$ part, respectively.
By Lemma~41 in~\cite{Patey2016Partial}, $A$ and $B$ are hyperimmune.
Therefore, by Lemma~\ref{lem:hyper-bifamily}, $\Hcal(A, B)$ is \hyper.
Any infinite ascending or descending sequence $G$ is a subset of $A$
or $B$, respectively. It follows that either $A$, or $B$ is not $G$-hyperimmune,
and by Lemma~\ref{lem:hyper-bifamily}, $\Hcal(A, B)$ is not $G$-\hyper.
It follows that $\sads$ does not preserve \propertyL.
\end{proof}

Whenever a Turing degree contains no $C$-hyperimmune set, it is said to be
\emph{$C$-hyperimmune-free}.
A Turing degree $\mathbf{d}$ is known to be $C$-hyperimmune-free if and only if
every function bounded by $\mathbf{d}$ is dominated by a $C$-computable function
(see Theorem III.3.8 in~\cite{Odifreddi1992Classical}).

\begin{lemma}\label{lem:hif-double-immune}
Let $\Hcal$ be a $C$-\hyper\ bifamily and $G$ be of $C$-hyperimmune-free degree.
Then $\Hcal$ is $C \oplus G$-\hyper.
\end{lemma}
\begin{proof}
Again, assume $C=\emptyset$ since the result relativizes.
Fix a bifamily $\Hcal$, and a set $G$ of hyperimmune-free degree
such that $\Hcal$ is not $ G$-\hyper. We want to show that $\Hcal$ is not \hyper.
Let $(\vec{E}, \vec{F})$ be a $ G$-computable biarray which does not meet $\Hcal$.
In particular, the function $f$ defined by $f(n, m) = \max E_n, F_{n,m}$ is $G$-computable,
and is therefore dominated by a computable function $g$.
Define the computable biarray $(\vec{K}, \vec{L})$ by $K_n = \{n+1, \dots, g(n, 0) \}$
and $L_{n,m} = \{m+1, \dots, g(n, m) \}$. It is easy to see that $E_n \subseteq K_n$
and $F_{n,m} \subseteq L_{n,m}$ for every $n, m \in \omega$. Indeed, $E_n > n$ and $\max E_n \leq f(n, 0) \leq g(n,0)$
so $E_n \subseteq \{n+1, \dots, g(n, 0) \}$. Similarly, $F_{n,m} > m$ and $\max F_{n,m} \leq f(n,m) \leq g(n,m)$,
so $F_{n,m} \subseteq \{m+1, \dots, g(n, m)\}$. Since $(\vec{E}, \vec{F})$ does not meet $\Hcal$,
$(E_n, F_{n,m}) \not \in \Hcal$, and by downward-closure of $\Hcal$
under the subset relation, $(K_n, L_{n,m}) \not \in \Hcal$. It follows that $(\vec{K}, \vec{L})$ does not meet $\Hcal$
and therefore that $\Hcal$ is not \hyper.
\end{proof}

\begin{corollary}\label{cor:wkl-double-immunity}
$\wkl$ preserves \propertyL.
\end{corollary}
\begin{proof}
Let $\Hcal$ be a  \hyper\ family, and let $T$ be an infinite  computable
binary tree. By the hyperimmune-free basis theorem~\cite{Jockusch197201}, there is an infinite
path $P \in [T]$ which is $C$-hyperimmune-free. By Lemma~\ref{lem:hif-double-immune},
$\Hcal$ is $  P$-\hyper.
\end{proof}

Given a bifamily $\Hcal$, let $\Bcal(\Hcal) \subseteq \omega^\omega$
be the closed set of all $X$ such that
for every $m, n \in \omega$, $X(\langle 0,n\rangle)$
and $X(\langle 1,n, m \rangle)$ are finite  sets $E_n > n$ and $F_{n,m} > m$
such that $(E_n, F_{n,m}) \not \in \Hcal$.

\begin{lemma}\label{lem:closed-set-double-immune}
A bifamily $\Hcal$ is $C$-\hyper\ if and only if $\Bcal(\Hcal)$ has no $C$-computable member.
\end{lemma}
\begin{proof}
The members of $\Bcal(\Hcal)$ are precisely the biarrays which fail meeting $\Hcal$.
The equivalence follows immediatly.
\end{proof}




\begin{corollary}\label{cor:coh-preserves-double-immunity}
$\coh$ preserves \propertyL.
\end{corollary}
\begin{proof}
Let $\Hcal$ be a  \hyper\ family, and let $\vec{R} = R_0, R_1, \dots$ be an infinite computable
sequence
\footnote{By a computable sequence of sets $R_0,R_1,\cdots$ we mean
$R_n$ is uniformly computable.} of sets. By Lemma~\ref{lem:closed-set-double-immune}, $\Bcal(\Hcal)$
has no  computable member. By~\cite[Corollary 2.9]{PateyCombinatorial},
there is an $\vec{R}$-cohesive set~$G$ such that $\Bcal(\Hcal)$ has no $ G$-computable member.
By Lemma~\ref{lem:closed-set-double-immune}, $\Hcal$ is $  G$-\hyper.
\end{proof}

\subsection{$\sts^2_3$ and $\semo$ do not preserve \propertyL}
Before proving that $\sts^2_3$ and $\semo$ do not preserve \propertyL, we must first introduce some notation.

Given a stable coloring $f : [\omega]^2 \to 3$ and two sets $E < F$,
we write $E \to_i F$ for $(\forall x \in E)(\forall y \in F)f(x, y) = i$.
For every $i< 3$, we let $C_i(f) = \{ x : (\forall^{\infty} y) f(x, y) = i \}$.
Finally, given a stable coloring $f : [\omega]^2 \to 3$, we let $\Hcal(f)$
be the bifamily of all pairs $(E, F)$ such that $E < F$,
$E \subseteq C_1(f)$, $F \subseteq C_2(f)$, and $E \to_0 F$.

\begin{proposition}\label{thm:sts23-priority}
There is a stable computable coloring $f : [\omega]^2 \to 3$ such that $\Hcal(f)$ is \hyper.
\end{proposition}
\begin{proof}
We build the coloring $f : [\omega]^2 \to 3$ by a finite injury priority argument.
For every $e \in \omega$, we want to satisfy the following requirement:
\begin{align}
\Rcal_e: &\text{If $\Phi_e$ is total, then there is some $n, m \in \omega$
such that}\\ \nonumber
&\text{$\Phi_e(n;1) \subseteq C_1(f)$, $\Phi_e(n,m;2) \subseteq C_2(f)$
and $\Phi_e(n;1) \to_0 \Phi_e(n,m;2)$.}
\end{align}
The requirements are given the usual priority ordering $\Rcal_0 < \Rcal_1 < \dots$
Initially, the requirements are neither partially, nor fully satisfied.
\begin{itemize}
	\item[(i)] A requirement $\Rcal_e$ \emph{requires a first attention at stage~$s$}
	if it is not partially satisfied and $\Phi_{e,s}(n;1) \downarrow = E$ for some set $E \subseteq \{e+1, \dots, s-1\}$
	such that no element in $E$ is restrained by a requirement of higher priority.
	If it receives attention, then it puts a restrain on $E$,
	commit the elements of $E$ to be in $C_0(f)$, and is declared \emph{partially satisfied}.

	\item[(ii)] A requirement $\Rcal_e$ \emph{requires a second attention at stage~$s$}
	if it is not fully satisfied, and $\Phi_{e,s}(n;1) \downarrow = E$ and $\Phi_{e,s}(n,m;2) \downarrow = F$
	for some sets $E, F \subseteq \{e+1, \dots, s-1\}$ such that $E \to_0 F$ and which are not restrained
	by a requirement of higher priority. If it receives attention,
	then it puts a restrain on $E \cup F$, commit the elements of $E$ to be in $C_1(f)$,
	the elements of $F$ to be in $C_2(f)$, and is declared \emph{fully satisfied}.
\end{itemize}
At stage~0, we let $f = \emptyset$.
Suppose that at stage $s$, we have defined $f(x, y)$ for every $x < y < s$.
For every $x < s$, if it is committed to be in some $C_i(f)$, set $f(x, s) = i$,
and otherwise set $f(x, s) = 0$.
Let $\Rcal_e$ be the requirement of highest priority which requires attention.
If $\Rcal_e$ requires a second attention,
then execute the second procedure, otherwise execute the first one.
In any case, reset all the requirements of lower priorities by setting them unsatisfied,
releasing all their restrains, and go to the next stage.
This completes the construction.
On easily sees by induction that each requirement acts finitely often, and is eventually
fully satisfied. This procedure also yields a stable coloring.
\end{proof}

\begin{corollary}
$\sts^2_3$ does not preserve \propertyL.
\end{corollary}
\begin{proof}
Let $f$ be the stable computable coloring of Proposition \ref{thm:sts23-priority}.
Let $G = \{ x_0 < x_1 < \dots \}$ be an infinite $f$-thin set.
We claim that $\Hcal(f)$ is not $G$-\hyper.
 Indeed, let $(\vec{E}, \vec{F})$ be the $G$-computable biarray
defined by $E_n = \{x_n\}$ and $F_{n,m} = \{x_{n+m}\}$. Fix any $n$. Suppose that $E_n \subseteq C_1(f)$
and $F_{n,m} \subseteq C_2(f)$ (if such $E_n,F_{n,m}$ does not exist then we are done).
In other words, for every sufficiently large $k$, $f(x_n, x_k) = 1$
and $f(x_{n+m}, x_k) = 2$. It follows that $G$ must be $f$-thin for color~0
\footnote{Given $f:[\omega]^n\rightarrow\omega$,
a set $G$ is $f$-thin for color $i $ iff $i\notin f[G]^n$.}, therefore $E_n \not \to_0 F_{n,m}$
and so $(E_n, F_{n,m}) \not \in \Hcal(f)$. The $G$-computable biarray $(\vec{E}, \vec{F})$ does not meet $\Hcal(f)$,
so $\Hcal(f)$ is not $G$-\hyper.
\end{proof}

\begin{corollary}
$\semo$ does not preserve \propertyL.
\end{corollary}
\begin{proof}
  Let $f : [\omega]^2\rightarrow 3$ be the stable computable coloring of Proposition \ref{thm:sts23-priority}.
Let $T$ be the stable computable tournament defined for every $x < y$ by $T(x, y)$ iff $f(x, y) = 1$.
Let $G = \{ x_0 < x_1 < \dots \}$ be an infinite transitive subtournament.
We claim that $\Hcal(f)$ is not $G$-\hyper. Indeed, let $(\vec{E}, \vec{F})$ be the $G$-computable biarray
defined by $E_n = \{x_n\}$ and $F_{n,m} = \{x_{n+m}\}$. Fix $n$. Suppose that $E_n \subseteq C_1(f)$
and $F_{n,m} \subseteq C_2(f)$ (if such $E_n,F_{n,m}$ does not exist then we are done). In other words, for every sufficiently large $k$, $f(x_n, x_k) = 1$
and $f(x_{n+m}, x_k) = 2$, so $T(x_n, x_k)$ and $T(x_k, x_{n+m})$ will hold.
By transitivity of $G$, $T(x_n, x_{n+m})$ must hold, so $f(x_n, x_{n+m}) = 1$. It follows that $E_n \not \to_0 F_{n,m}$
and so $(E_n, F_{n,m}) \not \in \Hcal(f)$. The $G$-computable biarray $(\vec{E}, \vec{F})$ does not meet $\Hcal(f)$,
so $\Hcal(f)$ is not $G$-\hyper.
\end{proof}

\subsection{$\ts^2_4$ preserves \propertyL}\label{subsect:ts24-preserves}
The purpose of this section is to prove the following theorem.

\begin{theorem}\label{thm:ts24-preserves-double-immunity}
$\ts^2_4$ preserves \propertyL.
\end{theorem}

This will be generalized to arbitrary tuples in the next section.
The notion of preservation of \propertyL\ for $\ts^{n+1}_k$ relates
to the notion of strong preservation of \propertyL\ for $\ts^n_k$ in the following sense.

\begin{lemma}\label{lem:ts-n-to-np1-strongly-double-immune}
Fix some $n \geq 1$ and $k \geq 2$.
If $\ts^n_k$ strongly preserves \propertyL,
then $\ts^{n+1}_k$ preserves \propertyL.
\end{lemma}
\begin{proof}
Let $\Hcal$ be a  \hyper\ family, and $f : [\omega]^{n+1} \to k$ be a  computable
instance of $\ts^{n+1}_k$. Let $\vec{R} = \langle R_{\sigma,i} : \sigma \in [\omega]^n, i < k \rangle$
be the  computable family of sets defined for every $\sigma \in [\omega]^n$ and $i < k$ by
$$
R_{\sigma,i} = \{ x \in \omega : f(\sigma,x) = i \}
$$
Since $\coh$ preserves \propertyL\ (Corollary~\ref{cor:coh-preserves-double-immunity}), there
is an $\vec{R}$-cohesive set~$G$ such that $\Hcal$ is $G$-\hyper.
Let $g : [\omega]^n \to k$ be the $\Delta^{0,G}_2$ instance of $\ts^n_k$ defined for every $\sigma \in [\omega]^n$ by
$$
g(\sigma) = \lim_{x \in G} f(\sigma, x)
$$
By strong preservation of $\ts^n_k$, there is an infinite $g$-thin set $H$
such that $\Hcal$ is $  G \oplus H$-\hyper.
By thinning out the set $H$, we obtain an infinite $ G \oplus H$-computable $f$-thin set $\tilde{H}$.
In particular, $\Hcal$ is $ \tilde{H}$-\hyper.
\end{proof}

It therefore remains to prove the following theorem.

\begin{theorem}\label{thm:ts14-strong-preserve-double-immunity}
$\ts^1_4$ strongly preserves \propertyL.
\end{theorem}
\begin{proof}
For notational convenience, we will prove the non-relativized version,
 which extends by routine arguments.
Let $\Hcal$ be a \hyper\ bifamily, and let $f:\omega\rightarrow 4$ be an arbitrary instance of $\ts^1_4$.
Without loss of generality, assume that there is no infinite subset~$H$ of $f^{-1}(i)$
such that $\Hcal$ is $  H$-\hyper\ (as otherwise we are done).
We are going to
\begin{align}\nonumber
\bullet\ \ &\text{build 4 infinite sets $(G_i:i<4)$
such that $G_i$ is $f$-thin for color $i$
and}\\ \nonumber
&\text{$\Hcal$ is $  G_i$-\hyper\ for some $i < 4$.}
\end{align}
We are going to build the sets~$G_i$ by a Mathias forcing whose \emph{conditions}
are tuples $(F_0, F_1, F_2, F_3, X)$, where $(F_i,X) $ is a Mathias condition,
 $F_i$ is $f$-thin for color $i$
and $\Hcal$ is $ X$-\hyper.
A condition $d = (E_0, E_1, E_2, E_3, Y)$ \emph{extends} $c = (F_0, F_1, F_2, F_3, X)$ (written $d \leq c$)
if $(E_i, Y)$ Mathias extends $(F_i, X)$ for every $i < 4$.

The first lemma ensures that every sufficiently generic filter
for this notion of forcing will induce  four 
 infinite sets.

\begin{lemma}\label{lem:ts14-strong-preserve-double-immunity-force-R}
For every condition $c = (F_0, F_1, F_2, F_3, X)$ and every $i < 4$,
there is an extension $d = (E_0, E_1, E_2, E_3, Y)$ of $c$ such that
$|E_i| > |F_i|$.
\end{lemma}
\begin{proof}
Fix $c$ and $i < 4$. Note that $X \setminus f^{-1}(i)$ is infinite,
since otherwise it
contradicts with our assumption that $\Hcal$ is $  X$-\hyper.
Let $x \in X \setminus f^{-1}(i)$ with $x > F_i$.
The condition $d = (E_0, E_1, E_2,E_3, X  )$
defined by $E_i = F_i \cup \{x\}$, and $E_j = F_j$ for $j \neq i$
is the desired extension of~$c$.
\end{proof}

A 4-tuple of sets $G_0, G_1, G_2, G_3$ \emph{satisfies} a condition
$c = (F_0, F_1, F_2, F_3, X)$ if $G_i$ satisfies the Mathias condition $(F_i, X)$
for every $i < 4$. A condition $c$ \emph{forces} a formula $\varphi(G_0, G_1, G_2, G_3)$
if the formula holds for every 4-tuple of sets $G_0, G_1, G_2, G_3$ satisfying $c$.
Given any $e_0, e_1, e_2, e_3$, we want to satisfy the following requirements:
$$
\Rcal_{e_0, e_1, e_2, e_3} \mbox{ : } \Rcal^0_{e_0} \vee \Rcal^1_{e_1} \vee \Rcal^2_{e_2} \vee \Rcal^3_{e_3}
$$
where $\Rcal^i_e$ is the requirement:
$$
\text{If $\Phi_e^{G_i}$ is total, then it meets $\Hcal$.}$$

\begin{lemma}\label{lem:ts14-strong-preserve-double-immunity-force-Q}
For every condition $c$ and every 4-tuple of indices $e_0, e_1, e_2, e_3$,
there is an extension $d$ of $c$ forcing $\Rcal_{e_0, e_1, e_2, e_3}$.
\end{lemma}
Lemma \ref{lem:ts14-strong-preserve-double-immunity-force-Q},
\ref{lem:ts-n-strongly-double-immune-assuming-force-R},
\ref{lem:fs1-left-trapped-preserves-2}
and
\ref{lem:fs2-left-trapped-preserves-2}
 are main technical lemmas of Theorem \ref{thm:many-and-ts24-not-sem-and-many}
 (where we prove a condition can be extended to force that a given Turing functional meets $\mcal{H}$).
 We briefly explain one of the ideas of these lemmas
 (which also appears in the main lemma \ref{uemlemmain} of Theorem \ref{uemth0})——a generalization of
 Seetapun forcing to build weak solution.
 Usually, given an arbitrary instance $f$ (of a problem),
 we want to build a solution  $G$ to $f$ so that $\Phi^G$ has  a desired behaviour.
Since $f$ is not computable, we cannot search computably among initial segments $F$ of solutions to $f$
(call such $F$ finite solution to $f$)
such that $\Phi^F$ has that behaviour.
The idea of Seetapun forcing is to find sufficiently many $F$
so that $\Phi^F$ has that behaviour. By ``sufficiently many", it means whatever $f$ looks like,
one of $F$ is, at least, a finite solution to $f$ .

In our case, this means we want to find sufficiently many $F$ such that
$\Phi^F(n;1)\downarrow$ and $\Phi^F(n,m;2)\downarrow$.
But this is not quite enough. It only gives (by compactness) two sets
$U_{n,m},V_{n,m}$ so that for every $g$, there is a finite solution $F$
of $g$
 such that $\Phi^F(n;1)\subseteq U_{n,m}\wedge \Phi^F(n,m;2)\subseteq V_{n,m}$.
That is, $U_{n,m}$ depends on $m$.
So one may want to try this: first, search for a sufficient collection $\mcal{F}$, so that
$\Phi^F(n;1)\downarrow$ for each $F\in\mcal{F}$;
second, search for a sufficient collection $\mcal{E}$ each $E\in\mcal{E}$ extends a member of $\mcal{F}$
and $\Phi^E(n,m;2)\downarrow$.
Let's see what sufficiency notion  $\mcal{F}$ should satisfy. When $\mcal{F}$ exists while $\mcal{E}$
does not, we have that for some instance $g$,
\begin{align}\label{fstseq12}
&\text{there is no finite solution $E$ of $g$
such that $E$ extends a member of $\mcal{F}$ }\\ \nonumber
&\text{and $\Phi^E(n,m;2)\downarrow$.}
\end{align}
We need to
find an  appropriate $F\in\mcal{F}$ such that $F$ is a finite solution to $g$ and
 restrict $G$ so that $G$ extends $F$ and $G$ is a solution to $g$
 (because by (\ref{fstseq12}), for such $G$, $\Phi^G(n,m;2)\uparrow$).
 Here ``appropriate" means: at least, $F$ is a finite solution to $f$.
 The sufficiency notion for $\mcal{F}$ should guarantee the existence of $F$.
This gives the following sufficiency notion:
 for every two instances $g$ and $h$, there is an $F\in\mcal{F}$ such that $F$
 is a finite solution to both $g,h$.
 This is exactly the sufficiency notion we use in Lemma \ref{lem:ts14-strong-preserve-double-immunity-force-Q}.
 For more complex problems, $F$ being a finite solution to $f$ is not enough, but
  we must ensure that imposing the restriction of $F$
 does not cut the candidate space too much. This concern gives rise
 to the more complex sufficiency notion (Definition \ref{def:tssuff}, \ref{def:fssuff}
 and Lemma \ref{lem:fs-with-enough-sets-free}, \ref{lem:tssuff}).

\begin{proof}
Fix $c = (F_0, F_1, F_2, F_3, X)$.
For notational convenience, we assume $X=\omega$ and $F_i = \emptyset$.
We  define 
 a partial computable biarray as follows.
 To obtain a desired extension of $c$, we take advantage
 of the failure of this biarray to meet $\Hcal$ or its non-totality.
\bigskip

\emph{Defining $U_n$}.
Given $n \in \omega$, search computably for some
finite set $U_n>n$\footnote{More concretely, we mean search the canonical index of $U_n$.} (if it exists) such that
for every pair of colorings
$g,h:\omega\rightarrow 4$, there are two colors $i_0 < i_1 < 4$
and two sets $E_{i_0}$ and $E_{i_1}$ such that for every $i \in \{i_0, i_1 \}$,
$E_i $ is both $g$-thin and $h$-thin for color $i$  and
\footnote{By compactness, if $U_n$ is found, there is a sufficient
sequence $(\mcal{F}_i:i<4)$
of finite  collections
  of finite sets so that for every $E\in\mcal{F}$,
$\Phi_{e_i}^{E}(n;1)\downarrow\subseteq U_n$.
Here sufficient means for every two colorings $g,h:\omega\rightarrow 4$,
there are $i_0<i_1<4$ and $E_i\in\mcal{F}_i$ for each $i\in \{i_0,i_1\}$
so that $E_i$ is $g$-thin, $h$-thin for color $i$. }
$$
\Phi_{e_i}^{   E_i}(n;1) \downarrow \subseteq U_n.
$$

\emph{Defining $V_{n,m}$}. Given $n, m \in \omega$, search  computably for some
finite  set $V_{n,m}>m$ (if it exists) such that for every
coloring $g:\omega\rightarrow 4 $, there is an $i < 4$ and a finite set $E_i $ $g$-thin for color $i$
such that
$$
\Phi_{e_i}^{  E_i}(n;1) \downarrow \subseteq U_n
	\wedge \Phi_{e_i}^{ E_i}(n, m;2) \downarrow \subseteq V_{n,m}.
$$
\smallskip

We now have multiple outcomes, depending on which of $U_n$ and $V_{n,m}$ is found.

\begin{itemize}
	\item Case 1: $U_n$ is not found for some $n \in \omega$. By compactness, the following $\Pi^{0}_1$ class
$\Pcal$ of  pairs of colorings
$g,h:\omega\rightarrow 4$
is nonempty:	
	 there are three indices $i_0 < i_1 < i_2 < 4$ such that for every $i \in \{i_0, i_1, i_2\}$
	and every finite set $E_i$ being both $g$-thin and $h$-thin for color $i$, we have
	$\Phi_{e_i}^{E_i}(n;1) \uparrow$.

	As $\wkl$ preserves \propertyL\
	(Corollary~\ref{cor:wkl-double-immunity}), there is a member $(g,h)$
 of $\Pcal$ such that
	$\Hcal$ is $g\oplus h$-\hyper.
	In particular, there are some $i_0 < i_1 < i_2 < 4$ such that for every $i \in \{i_0, i_1, i_2\}$
	and every  finite $E_i $ being $g$-thin and $h$-thin for color $i$, we have
	$\Phi_{e_i}^{E_i}(n;1) \uparrow$.
	There must be an $i\in \{i_0,i_1,i_2\}$ such that
the set
$Y=\{x:g(x)\ne i,h(x)\ne i\}$ is infinite.\footnote{This is where the argument works with $\ts^1_4$ and fails with $\ts^1_3$ : with $\ts^1_3$, we would only get two colors $\{i_0, i_1\}$, and if $g$ and $h$ are the constant functions $i_0$ and $i_1$, respectively, the set $Y$ is finite for every $i \in \{i_0, i_1\}$.}  Then clearly
	$(F_0, F_1, F_2, F_3, Y)$ is an extension of $c$.
For every $G$ satisfying $(F_i,Y)$, $G$ is $g$-thin for color $i$.
 Thus  $\Phi_{e_i}^G(n;1)\uparrow$. That is, $d$ forces $\Phi_{e_i}^G(n;1)\uparrow$,
 hence $\Rcal_{e_0, e_1, e_2, e_3}$.
	\bigskip

	\item Case 2: $U_n$ is found, but not $V_{n,m}$ for some $n, m \in \omega$. By compactness,
	the  following $\Pi^{0}_1$ class
$\Pcal$ of    colorings $g:\omega\rightarrow 4$  is nonempty:
	for every $i < 4$ and every finite set $E_i $ $g$-thin for color $i$,
\begin{align}\label{fstseq11}
	 \Phi_{e_i}^{E_i}(n;1) \downarrow \subseteq U_n
		\Rightarrow \Phi_{e_i}^{ E_i}(n, m;2) \uparrow.
	\end{align}
 As $\wkl$ preserves \propertyL\
	(Corollary~\ref{cor:wkl-double-immunity}), there is a member $g $
	of $\Pcal$ such that $\Hcal$ is $g$-\hyper.
	By definition of $U_n$ (where we take $h=f$ in the definition of
$U_n$),
	there are some $i_0 < i_1 < 4$ and some finite sets $E_{i_0}$ and $E_{i_1}$
	such that for every  $i \in \{i_0, i_1\}$, $E_i $ is both $g$-thin and $f$-thin for color $i$
and
	$$
	\Phi_{e_i}^{  E_i}(n;1) \downarrow \subseteq U_n.
	$$
	In particular, there must be some $i \in \{i_0, i_1\}$ such that  the set $Y= \{x :g(x)\ne i\}$ is infinite.
	Consider the extension  $d = (D_0, D_1, D_2,D_3, Y)$  of $c$ defined by
	$D_i = F_i \cup E_i$ and $D_j = F_j$ for each $j \neq i$.
To see $d$
	forces
	$\Phi_{e_i}^{  G_i}(n, m;2) \uparrow$ (hence forces $\Rcal_{e_0, e_1, e_2, e_3}$),
note that for every $G$ satisfying $(D_i,Y)$,
$G$ is $g$-thin for color $i$.
But $\Phi_{e_i}^{D_i}(n;1)\downarrow\subseteq U_n$.
Thus, by definition of $g$ (namely (\ref{fstseq11})), $\Phi_{e_i}^{  G}(n, m;2) \uparrow$.
	\bigskip

	\item Case 3: $U_n$ and $V_{n,m}$ are found for every $n,m \in \omega$.
	By  \propertyL\ of $\Hcal$, there is some $n, m \in \omega$
	such that $(U_n, V_{n,m}) \in \Hcal$.
	In particular, by definition of $V_{n,m}$ (where we take $g=f$ in the definition of
$V_{n,m}$), there is some $i < 4$ and some finite set $E_i  $ such that
$E_i$ is $f$-thin for color $i$ and
	$$
	 \Phi_{e_i}^{ E_i }(n;1) \downarrow \subseteq U_n
		\wedge \Phi_{e_i}^{  E_i }(n, m;2) \downarrow\subseteq V_{n,m}.
	$$
	The condition $(D_0, D_1, D_2, D_3, X)$ defined by $D_i = F_i \cup E_i$
	and $D_j = F_j$ for each $j \neq i$ is an extension of $c$ forcing $\Rcal_{e_0, e_1, e_2, e_3}$.
\end{itemize}
This completes the proof of Lemma~\ref{lem:ts14-strong-preserve-double-immunity-force-Q}.
\end{proof}

Let $\Fcal = \{c_0, c_1, \dots \}$ be a sufficiently generic filter for this notion of forcing,
where $c_s = (F_{0,s}, F_{1,s}, F_{2,s}, F_{3,s},\allowbreak X_s)$, and let $G_i = \bigcup_s F_{i,s}$.
By definition of a condition, for every $i < 4$, $G_i $ is $f$-thin for color $i$.
By Lemma~\ref{lem:ts14-strong-preserve-double-immunity-force-R}, $G_i$ are all infinite,
and by Lemma~\ref{lem:ts14-strong-preserve-double-immunity-force-Q}, there is some $i < 4$ such that $\Hcal$ is $  G_i$-\hyper.
This completes the proof of Theorem~\ref{thm:ts14-strong-preserve-double-immunity}.
\end{proof}

For $\msf{\ts^2_3}$, its relation with $\msf{EM}$ is unclear.
\begin{question}
Does $\msf{\ts^2_3}$ imply $\msf{EM}$?
\end{question}

\subsection{Generalized cohesiveness preserves \propertyL}

In order to  prove that $\ts^n_k$ and $\fs^n$ preserve \propertyL\ for sufficiently
large $k$, we first need to prove the following technical theorem, which thin out colors
while preserving \propertyL.
The proof is a slight adaptation of~\cite{PateyCombinatorial} to \propertyL.
We however reprove it for the sake of completeness. We will need the case $t = n-1$
for $\ts^n_k$, and the case $t = n$ for $\fs^n$.
Fix a set  $C$, a bifamily $\Hcal$ which is $C$-\hyper,
an infinite set $X\leq_T C$; let $f : [\omega]^n \to k$ be a coloring.
\begin{proposition}\label{thm:gen-coh-double-immmunity}
Assume $\ts^s_{d_s+1}$ strongly preserves \propertyL\ for each  $0<s<n$.
Then there exists an infinite set $G\subseteq X$ such that $\Hcal$ is $C \oplus G$-\hyper,
and for every $\sigma \in [\omega]^{<\omega}$ such that $0< \card{\sigma} < n$,
there is a set $I_\sigma \subseteq \{0, \dots, k-1\}$ such that $|I_\sigma| \leq d_{n-|\sigma|}$ and
$$
(\exists b)(\forall \tau \in [G \cap (b, +\infty)]^{n-|\sigma|}) f(\sigma, \tau) \in I_\sigma.
$$
\end{proposition}
\begin{proof}
For notational convenience, assume $X=\omega$ and $C = \emptyset$.
Our forcing conditions are Mathias conditions $(F, Y)$ where $\Hcal$ is $  Y$-\hyper.
The first lemma shows that $\Hcal$ will be $  G$-\hyper\ for every sufficiently generic filter.
Given $e$, let $\Rcal_e$ be the requirement:
\begin{align}\nonumber
\text{If $\Phi^{ G}_e$ is total, then it meets $\Hcal$.}
\end{align}

\begin{lemma}\label{lem:mathias-double-immunity}
Given a condition $c = (F, Y)$ and an index $e \in \omega$,
there is an extension $d$ forcing $\Rcal_e$.
\end{lemma}
\begin{proof}
This simply follows by a finite extension argument.
Again for notational convenience, assume $F=\emptyset$ and $Y=\omega$.
 We define a partial  computable biarray as follows.
\bigskip

\emph{Defining $U_n$}.
Given $n \in \omega$, search  computably for some
finite   set $U_n>n$ (if it exists) and a finite set $E  $
such that
$$
\Phi_e^{   E}(n;1) \downarrow = U_n.
$$

\emph{Defining $V_{n,m}$}. Given $n, m \in \omega$, search computably for some
finite set $V_{n,m}>m$ (if it exists) and a finite set $E  $ such that
$$
\Phi_e^{  E}(n;1) \downarrow = U_n
	\wedge \Phi_e^{  E}(n, m;2) \downarrow = V_{n,m}.
$$
\smallskip

We now have multiple outcomes, depending on which $U_n$ and $V_{n,m}$ is found.
\begin{itemize}
	\item Case 1: $U_n$ is not found for some $n \in \omega$.
	Then the condition $c = (F, Y)$ already forces  $\Phi_e^{ G}(n;1) \uparrow$
	and therefore forces $\Rcal_e$.

	\item Case 2: $U_n$ is found, but not $V_{n,m}$ for some $n, m \in \omega$.
	By definition of $U_n$, there is a finite set $E  $ such that
	$$
	\Phi_e^{   E }(n;1) \downarrow = U_n
	$$
	The condition $d = ( E, Y)$ is an extension of $c$
	forcing  $\Phi_e^{  G}(n,m; 2) \uparrow$.

	\item Case 3: $U_n$ and $V_{n,m}$ are found for every $n,m \in \omega$.
	By  \propertyL\ of $\Hcal$, there is some $n, m \in \omega$
	such that $(U_n, V_{n,m}) \in \Hcal$.
	In particular, by definition of $V_{n,m}$, there is a finite set $E  $ such that
	$$
	\Phi_e^{  E}(n;1) \downarrow = U_n
		\wedge \Phi_e^{  E}(n, m;2) \downarrow = V_{n,m}
	$$
	The condition $d = (  E, Y )$ is an extension of $c$ forcing
	$\Phi^{  G}_e$ to meet $\Hcal$, and therefore forcing $\Rcal_e$.
\end{itemize}
This completes the proof of Lemma~\ref{lem:mathias-double-immunity}.
\end{proof}

\begin{lemma}\label{lem:gen-coh-function}
For every condition $(F, Y)$ and $\sigma \in [\omega]^{<\omega}$ such that $0< \card{\sigma} < n$,
there is a finite set $I \subseteq \{0, \dots, k-1\}$ with $|I| \leq d_{n-|\sigma|}$
and an extension $(F, \t Y)$ such that
$$
(\forall \tau \in [\tilde{X}]^{n-|\sigma|})f(\sigma, \tau) \in I.
$$
\end{lemma}
\begin{proof}
This simply follows from strong preservation of $\ts^{n - |\sigma|}_{d_{n-|\sigma|}+1}$ .
Define the function $g : [Y]^{n - |\sigma|} \to k$ by $g(\tau) = f(\sigma, \tau)$.
Since  $\ts^{n - |\sigma|}_{d_{n-|\sigma|}+1}$ strongly preserves \propertyL
\ (since $0<n-|\sigma|<n$),
there exists an infinite  $\t Y\subseteq Y$ and a finite set $I \subseteq \{0, \dots, k-1\}$
such that $\Hcal$ is $  \tilde{Y}$-\hyper, $|I| \leq d_{n-|\sigma|}$,
and $(\forall \tau \in [\tilde{Y}]^{n-|\sigma|})f(\sigma, \tau) \in I$.
The condition $(F, \tilde{Y})$ is the desired extension.
\end{proof}

Let~$\Fcal = \{c_0, c_1, \dots\}$ be a sufficiently generic filter containing $(\emptyset, \omega)$,
where $c_s = (F_s, X_s)$. The filter~$\Fcal$ yields a unique infinite set~$G = \bigcup_s F_s$.
By Lemma~\ref{lem:mathias-double-immunity}, $\Hcal$ is $  G$-\hyper.
By Lemma~\ref{lem:gen-coh-function}, $G$ satisfies the property of the theorem.
This completes the proof of Proposition~\ref{thm:gen-coh-double-immmunity}.
\end{proof}

\subsection{$\ts^n$ preserves \propertyL}\label{subsect:tsn-preserves}
\def\p{2d_{n-1}+\sum_{0\leq s<n}d_{s-1}d_{n-s-1}}
Define the sequence $d_0, d_1, \dots$ by induction as follows:
$$
d_0 = 1 \hspace{1cm} d_{n} = \p+\sum_{0 < s < n} d_s d_{n-s} \hspace{20pt} \mbox{ for } n > 1
$$
The purpose of this section is to prove the following theorem.

\begin{theorem}\label{thm:ts-n-strongly-preserves-double}
$\ts^n_{d_n+1}$ strongly preserves \propertyL\ for every $n \geq 1$.
\end{theorem}
\begin{proof}
We prove by induction over $n \geq 1$ that
$\ts^n_{d_{n-1}+1}$ preserves \propertyL, and that $\ts^n_{d_n+1}$ strongly preserves \propertyL.
If $n = 1$, $\ts^1_2$ is a computably true statement, that is, every instance
has a solution computable in the instance, so $\ts^1_2$ preserves \propertyL.
On the other hand, $\ts^1_4$ strongly preserves \propertyL\ follow from Theorem \ref{thm:ts14-strong-preserve-double-immunity}.
If $n > 1$, then by the induction hypothesis, $\ts^{n-1}_{d_{n-1}+1}$ strongly preserves \propertyL,
so by Lemma~\ref{lem:ts-n-to-np1-strongly-double-immune}, $\ts^n_{d_{n-1}+1}$ preserves \propertyL.
Assuming by  the induction hypothesis that $\ts^s_{d_s+1}$ strongly preserves \propertyL\
for every $0 < s < n$, and that $\ts^n_{d_{n-1}+1}$ preserves \propertyL,
by Theorem~\ref{thm:ts-n-strongly-double-immune-assuming}, $\ts^n_{d_n+1}$ strongly preserves \propertyL.
\end{proof}

We need to prove Theorem~\ref{thm:ts-n-strongly-double-immune-assuming} to complete the proof of
 Theorem~\ref{thm:ts-n-strongly-preserves-double}.
We start with the following technical lemma which thins out colors while
preserving \hyper.
Fix a set  $C$, a bifamily $\Hcal$
which is $C$-\hyper, an infinite set $X\leq_T C$ and a coloring $f : [\omega]^n \to k$.
\begin{lemma}\label{thm:gen-coh-art}
Assume $\ts^s_{d_s+1}$ strongly
preserves \propertyL\ for every  $0<s <n$.
Then there is an infinite set $Y\subseteq X$ so that $\Hcal$ is $C \oplus Y$-\hyper,
and a finite set $I \subseteq \{0, \dots, k-1\}$ with
$
|I| \leq \sum_{0 < s < n} d_s d_{n-s}
$
such that for each  $0<s <n$,
$$
(\forall \sigma \in [Y]^s)(\exists b)(\forall \tau \in [Y \cap (b, \infty)]^{n-s})f(\sigma, \tau) \in I
$$
\end{lemma}
\begin{proof}
For notational convenience, assume $C=\emptyset$.
Apply Proposition \ref{thm:gen-coh-double-immmunity}
to get an infinite set $X_0\subseteq X$
so that $\Hcal$
is $X_0  $-\hyper\ and for every $\sigma\in [\omega]^{<\omega}$
with $0<|\sigma|<n$,
there is an $I_\sigma$
such that  $|I_\sigma| \leq d_{n-|\sigma|}$ and
$$
(\exists b)(\forall \tau \in [X_0 \cap (b, +\infty)]^{n-|\sigma|}) f(\sigma, \tau) \in I_\sigma.
$$

For each  $0<s<n$
 and~$\sigma \in [\omega]^s$, let $F_s(\sigma) = I_\sigma$.
Since $\ts^s_{d_s+1}$ strongly preserves \propertyL,
  for each  $0<s<n$,
  there is an infinite set $Y\subseteq X_0$
  such that $\Hcal$ is $Y $-\hyper\ and
  such that $|F_s[Y]^s|\leq d_s$
  for all $0<s<n$.
  Let $\Ical_s = F_s[Y]^s$ for each  $0<s<n$, and let
$I = \bigcup_{J\in\Ical_s, 0 < s < n} J$.
Then
$$
|I| \leq \sum_{0 < s < n} d_s d_{n-s}.
$$
We now check that the property is satisfied.
Fix an $0<s<n$, a $\sigma \in [Y]^s$ and let $b\in\omega$ be sufficiently large. Because  $Y \subseteq X_0$,
$$
 (\forall \tau \in [Y \cap (b,+\infty)]^{n-s}) f(\sigma,\tau) \in I_\sigma.
$$
So $F_s(\sigma) = I_\sigma$, but $\sigma \in [Y]^s$, hence $I_\sigma \in \Ical_s$.
It follows that
$$
 (\forall \tau \in [Y \cap (b,+\infty)]^{n-s}) f(\sigma,\tau) \in I .
$$
 This completes the proof.
\end{proof}

We need to prove a second lemma saying that if we have sufficiently many finite thin sets,
one of them can be extended to an infinite one.
This argument is a generalization of case 2 of Lemma
\ref{lem:ts14-strong-preserve-double-immunity-force-Q}.
\begin{definition}[\tssufficient]
\label{def:tssuff}
Let $(\mcal{F}_i:i<p)$ be a $p$-tuple of finite collections
of finite sets.
We say $(\mcal{F}_i:i<p)$
is $n$-\emph{\tssufficient}
iff for every
sequence of colorings $(f_{s,j}:[\omega]^s\rightarrow  d_n+1)_{0\leq s<n, j< d_{n-s-1}}$,
there is an $i<p$, an $F\in\mcal{F}_i$ such that
$F$ is $f_{s,j}$-thin for color $i$
for all $0\leq s<n, j<d_{n-s-1}$.
\end{definition}
In our application, $p$ will be smaller than $d_n$.
Let $(\mcal{F}_i:i<p)$ be a $p$-tuple of finite collections
of finite  sets.
\begin{lemma}\label{lem:tssuff}
Assume $\ts^s_{d_s+1}$ strongly
preserves \propertyL\ for every  $0<s <n$.
Suppose $f\leq_T C$, $(\mcal{F}_i:i<p)$
is $n$-\tssufficient\ and for every $i < p$, $ \mcal{F}_i$ is $f$-thin for color~$i$\footnote{Each member of $\mcal{F}_i$ is $f$-thin for color $i$.} .
Then there exists an $i<p$, an $F\in\mcal{F}_i$
and an infinite set $Y\subseteq X$
such that $F\cup Y$ is $f$-thin for color $i$ and $\Hcal$ is $Y\oplus C$-\hyper.
\end{lemma}
\begin{proof}
Again, for notational convenience, assume $C=\emptyset$.
Let $E=\cup_{F\in \mcal{F}_i,i<p}F$.
For every $s<n$, every $\sigma\in [E]^s$,
let coloring $f_\sigma:[\omega]^{n-s}\rightarrow  d_n+1$ be defined as $f_\sigma(\tau) = f(\sigma,\tau)$.
By Lemma~\ref{lem:ts-n-to-np1-strongly-double-immune}, $\ts^s_{d_{s-1}+1}$ admits preservation of \propertyL\ for $0\leq s\leq n$ (set $d_{-1} = 1$), so
there is an infinite set $Y\subseteq X$ with $\Hcal$ being $Y $-\hyper\ such that
for every $0\leq s<n$, every $\sigma\in [E]^s$,
there is a $I_\sigma$ with $|I_\sigma|\leq d_{n -s-1}$
such that $$f_\sigma[Y]^{n-s}\subseteq I_\sigma.$$
For every $0\leq s<n$ and $j< d_{n-s-1}$, let $f_{s,j}$ be the coloring on
$[E]^s$ such that $f_{s,j}(\sigma)  $ is the $j$th element of $I_\sigma$.

By $n$-\tssufficient\ of $(\mcal{F}_i:i<p)$,
there is a $i<p$, an $F\in\mcal{F}_i$ such that
$F$ is $f_{s,j}$-thin for color $i$ for all $0\leq s<n$ and $j<d_{n-s-1}$.
In particular, $i\notin I_\emptyset $ since $F$ is $f_{0,j}$-thin for color $i$
(and $f_{0,j}\equiv$  the $j$th element of $I_\emptyset$).
This means $Y$ is $f$-thin for color $i$.

We show that $F\cup Y$ is $f$-thin for color $i$.
To see this, let $\sigma\in [F]^{<\omega},\tau\in [Y]^{<\omega}$
with $|\sigma\cup \tau| = n$.
When $|\sigma|=n$ or $|\tau| = n$, $ f(\sigma,\tau)\ne i$
follows from $f$-thin for color $i$ of $F$ and $Y$ respectively.
When $|\sigma|=s$ with $0<s<n$, since $F$ is $f_{s,j}$-thin for color $i$ for all $j<d_{n-s-1}$,
we have $f_{s,j}(\sigma)\ne i$ for all $j<d_{n-s-1}$.
This means $i\notin I_\sigma$.
Thus $f(\sigma,\tau)\ne i$ since $f(\sigma,\tau)\in I_\sigma$
(by choice of $Y$).
\end{proof}

We are now ready to prove the missing theorem.

\begin{theorem}\label{thm:ts-n-strongly-double-immune-assuming}
Fix some $n \geq 2$, and suppose that $\ts^s_{d_s+1}$ strongly preserves \propertyL\
for every $0 < s < n$, and that $\ts^n_{d_{n-1}+1}$ preserves \propertyL.
Then  $\ts^n_{d_n+1}$ strongly preserves \propertyL.
\end{theorem}
\begin{proof}
Fix a coloring $f : [\omega]^n \to d_n+1$,
  and a bifamily $\Hcal$ which is  \hyper.
Let~$q = \sum_{0 < s < n} d_s d_{n-s}$.
By Lemma~\ref{thm:gen-coh-art}, we assume
that there exists a finite set $I_f$ of cardinality $q$
 such that for every  $0<s<n$,
\begin{align}\label{fstseq5}
(\forall \sigma \in [\omega]^s)(\exists b)
	(\forall \tau \in [ \omega\cap (b,+\infty)]^{n-s}) f(\sigma, \tau) \in I_f.
\end{align}

Let $p =  1+\p$, so $d_n= p+q-1$.
By renaming the colors of $f$, we can assume without loss of generality
that $I_f = \{p, p+1, \dots, d_n\}$.
We will construct simultaneously $p$ infinite sets $G_0, \dots, G_{p -1}$ such that
$\Hcal$ is $  G_i$-\hyper\ for some $i < p$. We furthermore ensure that for each $i < p$,
$G_i$ is $f$-thin for color~$i$.
We construct our sets $G_0, \dots, G_{p-1}$ by a Mathias forcing whose \emph{conditions} are tuples
$(F_0, \dots, F_{p-1}, X)$, where $(F_i, X)$ is a Mathias condition for each $i < p$
and the following properties hold:
\begin{itemize}
	\item[(a)] $(\forall \sigma \in [F_i]^s)
	(\forall \tau \in [F_i \cup X]^{n-s}) f(\sigma, \tau) \geq p$ for every $0<s<n$.
\item [(b)] $F_i$ is $f$-thin for color $i$ for every $i < p$.
	\item[(c)] $\Hcal$ is $  X$-\hyper.
\end{itemize}
 A \emph{precondition} is a tuple of  Mathias conditions satisfying (b) and (c).
A precondition $d = (E_0, \dots, E_{p-1}, Y)$ \emph{extends}
a precondition $c = (F_0, \dots, F_{p-1}, X)$ (written $d \leq c$)
if $(E_i, Y)$ Mathias extends $(F_i, X)$ for each $i < p$.
Obviously, $(\emptyset, \dots, \emptyset, \omega)$ is a  condition.
Therefore, the partial order is non-empty.
We note the following simple properties
of conditions.

\begin{lemma}\label{lem:ts-n-strongly-double-immune-assuming-combi}
\
\begin{enumerate}
\item Every precondition can be extended to a condition.
\item For every condition $c = (F_0, \dots, F_{p-1}, X)$,
every $i < p$ and every finite set $E \subseteq X$ $f$-thin for color $i$,
  $d = (F_0, \dots, F_{i-1}, F_i \cup E, F_{i+1}, \dots, F_{p-1}, X)$
is a precondition extending $c$.
\end{enumerate}
\end{lemma}
\begin{proof}
Item (1) is trivial by (\ref{fstseq5}).
For item (2),
by property (a) and (b) for condition $c$ and by $f$-thin for color $i$ of $E$,
we have $i\notin f[F_i\cup E]^n$,
so $d$ is a precondition.
\end{proof}

The next lemma states that every sufficiently generic filter yields
infinite sets $G_0, \dots, G_{p-1}$.

\begin{lemma}\label{lem:ts-n-strongly-double-immune-assuming-force-Q}
For every condition $c = (F_0, \dots, F_{p-1}, X)$ and every $i < p$,
there is an extension $d = (E_0, \dots, E_{p-1}, X)$ of $c$ such that
$|E_i| > |F_i|$.
\end{lemma}
\begin{proof}
Fix $c$ and some $i < p$, and let $x\in X\setminus F_i$.
In particular, $[x]^n = \emptyset$, so $i \not \in f[x]^n$.
Thus, by Lemma~\ref{lem:ts-n-strongly-double-immune-assuming-combi},
there is an extension $d = (E_0, \dots, E_{p-1}, X)$ of $c$ such that
$E_i = F_i \cup \{x\}$.
\end{proof}

A $p$-tuple of sets $G_0, \dots, G_{p-1}$ \emph{satisfies} a condition
$c = (F_0, \dots, F_{p-1}, X)$ if $G_i$ satisfies the Mathias condition $(F_i, X)$.
A condition $c$ \emph{forces} a formula $\varphi(G_0, \dots, G_{p-1})$
if the formula holds for every $p$-tuple of sets $G_0, \dots, G_{p-1}$ satisfying $c$.
For every $e_0, \dots, e_{p-1} \in \omega$, we want to satisfy the following requirement
$$
\Rcal_{e_0, \dots, e_{p-1}} : \Rcal_{e_0} \vee \dots \vee \Rcal_{e_{p-1}}
$$
where $\Rcal_{e_i}$ is the requirement
\begin{align}\nonumber
\text{If $\Phi^{  G_i}_{e_i}$ is a total,
then $\Phi^{  G_i}_{e_i}$ meets $\Hcal$.}
\end{align}

\begin{lemma}\label{lem:ts-n-strongly-double-immune-assuming-force-R}
For every condition $c$ and every $p$-tuple of indices $e_0, \dots, e_{p-1}$,
there is an extension $d$ of $c$ forcing $\Rcal_{e_0, \dots, e_{p-1}}$.
\end{lemma}
\begin{proof}
Fix $c = (F_0, \dots, F_{p-1}, X)$.
By Lemma \ref{lem:ts-n-strongly-double-immune-assuming-combi},
for notational convenience, we assume $F_i=\emptyset$ and $X=\omega$
\footnote{More specifically, if we can always extends a condition of form $(\emptyset,\cdots,\emptyset,X)$,
then given a condition $(F_0, \dots, F_{p-1}, X)$,
we can find a desired extension $(E_0,\cdots,E_{p-1})$ of $(\emptyset,\cdots,\emptyset,X)$.
But $(F_0\cup E_0,\cdots,F_{p-1}\cup E_{p-1},Y)$ is a precondition by Lemma \ref{lem:ts-n-strongly-double-immune-assuming-combi}.}.
We define a partial computable biarray as follows.
\bigskip

\emph{Defining $U_n$}.
Given $r \in \omega$, search  computably for some
finite  set $U_r>r$ (if it exists) such that
for every pair of colorings $g, h : [\omega]^n \to d_n+1$,
there
is a $n$-\tssufficient\ $p$-tuple $(\mcal{E}_i:i<p) $ of
finite collections of finite sets
with $\mcal{E}_i$ being $g$-thin, $h$-thin   for color $i$
such that for every $i<p$, every $E\in\mcal{E}_i$, we have
$$
\Phi_{e_i}^{  E }(r;1) \downarrow \subseteq U_r.
$$

\emph{Defining $V_{r,m}$}. Given $r, m \in \omega$, search  computably for some
finite  set $V_{r,m}>m$ (if it exists) such that for every
coloring $g : [\omega]^n \to d_n+1$, there is some $i < p$ and some $E_i  $
$g$-thin for color~$i$ such that
$$
\Phi_{e_i}^{  E_i }(r;1) \downarrow \subseteq U_r
	\wedge \Phi_{e_i}^{E_i }(r, m;2) \downarrow \subseteq V_{r,m}.
$$
\smallskip

We now have multiple outcomes, depending on which $U_r$ and $V_{r,m}$ is found.

\begin{itemize}
	\item Case 1: $U_r$ is not found for some $r \in \omega$. By compactness, the following $\Pi^{0}_1$ class
	$\Pcal$ of   pairs of colorings $g, h : [\omega]^n \to d_n+1$ is nonempty:
	 there
is no $n$-\tssufficient\ $(\mcal{E}_i:i<p)$
finite collections of finite sets
such that $\mcal{E}_i$ is both $g$-thin, $h$-thin for color $i$
and for every $i<p,E\in\mcal{E}_i$, we have
	$
	\Phi_{e_i}^{ E}(r;1) \downarrow.
	$

	As $\wkl$ preserves \propertyL\
	(Corollary~\ref{cor:wkl-double-immunity}), there is a member  $g, h $ of
	  $\Pcal$ such that
	$\Hcal$ is $g \oplus h  $-\hyper.
Unfolding the definition of $n$-\tssufficient\ and using compactness,
the following $\Pi_1^{0,g\oplus h}$ class $\mcal{Q}$ of
sequence $(f_{s,j}:[\omega]^s\rightarrow  d_n+1)_{0\leq s<n,j<d_{n-s-1}}$	
of colorings is nonempty:
for every $i<p$, every finite set $E$ which is both $g$-thin, $h$-thin for color $i$
and is $f_{s,j}$-thin for color $i$ for all $0\leq s<n,j<d_{n-s-1}$, we have
$
	\Phi_{e_i}^{ E}(r;1) \uparrow.
	$

As $\wkl$ preserves \propertyL\
	(Corollary~\ref{cor:wkl-double-immunity}), there is a member  $(f_{s,j}:0\leq s<n,j<d_{n-s-1})$	 of
	  $\mcal{Q}$ such that
$\Hcal$ is $g\oplus h\oplus_{0\leq s<n,j<d_{n-s-1}}f_{s,j}$-\hyper.
Since $\ts^s_{d_{s-1}+1}$ preserves \propertyL\ for all $0\leq s\leq n$,
there is an infinite set $Y$ such that
$|f_{s,j}[Y]^s|\leq d_{s-1}$ for all $0\leq s<n,j<d_{n-s-1}$, $|g[Y]^s|,|h[Y]^s|\leq d_{n-1}$
and $\Hcal$ is $Y$-\hyper.
Since $p> \p$,
there is an $i<p$ such that $Y$ is $f_{s,j}$-thin for color $i$
for all $0\leq s<n,j<d_{n-s-1}$ and both $g$-thin, $h$-thin for color $i$.

Clearly $d = (F_0, \dots, F_{p-1}, Y)$ is an extension of $c$
\footnote{When we say ``extension of $c$", we mean a precondition extending $c$.}.
We prove that
$d$ forces $\Rcal_{e_0, \dots, e_{p-1}}$.
This is because if $G_i$ satisfies $(F_i,Y)$, then $G_i$ is both $g$-thin and $h$-thin for color $i$
and $f_{s,j}$-thin for color $i$ for all $0\leq s<n,j<d_{n-s-1}$.
Thus, by definition of $g,h,(f_{s,j}:0\leq s<n,j<d_{n-s-1}) $,
we have $\Phi_{e_i}^{ G_i}(r;1) \uparrow$.

	\bigskip

	\item Case 2: $U_r$ is found, but not $V_{r,m}$ for some $r, m \in \omega$. By compactness,
	the following $\Pi^{0}_1$ class $\Pcal$ of colorings $g : [\omega]^n \to d_n+1$ is nonempty:
	 for every $i < p$ and every $E_i  $ $g$-thin for color~$i$,
\begin{align}\label{fsts-tsnstrongpreserve-defofg}
	\Phi_{e_i}^{  E_i }(r;1) \downarrow \subseteq U_r
		\Rightarrow \Phi_{e_i}^{  E_i }(r, m;2) \uparrow.
	\end{align}
	  As $\wkl$ preserves \propertyL\
	(Corollary~\ref{cor:wkl-double-immunity}), there is a member $g  $
of $\Pcal$ such that $\Hcal$ is $g  $-\hyper.
	By definition of $U_r$ (where we take $h = f$), there
is a $n$-\tssufficient\ $p$-tuple $(\mcal{E}_i:i<p)$ of
finite collections of finite sets
such that $\mcal{E}_i$ is both $g$-thin and $f$-thin for color $i$
and for every $i<p$, every $ E\in\mcal{E}_i$, we have
	$$
	\Phi_{e_i}^{ E}(r;1) \downarrow \subseteq U_r.
	$$
By Lemma \ref{lem:tssuff}, there is an $i<p$, $E\in\mcal{E}_i$
and an infinite set $Y$ such that $E\cup Y$ is $g$-thin for color $i$
and $\Hcal$ is $Y$-\hyper.
Consider the precondition $(F_0,\cdots,F_{i-1},F_i\cup E,F_{i+1},\cdots F_{p-1},Y)$.

It remains to show that $d$ forces $\Phi_{e_i}^{G_i}(n,m;2)\uparrow$.
This is because if $G_i$ satisfies $(E,  Y)$, then $G_i$ is $g$-thin  for color $i$.
But $\Phi_{e_i}^{ E}(r;1) \downarrow \subseteq U_r$.
 Thus, by definition of $g$
 (namely (\ref{fsts-tsnstrongpreserve-defofg})),
  $\Phi_{e_i}^{G_i}(r,m;2) \uparrow$.

	\bigskip

	\item Case 3: $U_r$ and $V_{r,m}$ are found for every $r,m \in \omega$.
	By  \propertyL\ of $\Hcal$, there is some $r, m \in \omega$
	such that $(U_r, V_{r,m}) \in \Hcal$.
	In particular, by definition of $V_{n,m}$ (where we take $g=f$),
	there is some $i < p$ and some $E_i $ $f$-thin for color~$i$ such that
	$$
	\Phi_{e_i}^{  E_i }(r;1) \downarrow \subseteq U_r
		\wedge \Phi_{e_i}^{E_i }(r, m;2) \downarrow \subseteq V_{r,m}
	$$
	
Since $i\notin f[E_i]^n$,
we have $d = (F_0, \dots, F_{i-1}, F_i \cup E_i, F_{i+1}, \dots, F_{p-1}, X)$ is an extension of $c$.
Clearly  $d$ forces $\Rcal_{e_0, \dots, e_{p-1}}$.
\end{itemize}
This completes the proof of Lemma~\ref{lem:ts-n-strongly-double-immune-assuming-force-R}.
\end{proof}

Let $\Fcal = \{c_0, c_1, \dots \}$ be a sufficiently generic filter for this notion of forcing,
where $c_s = (F_{0,s}, \dots, F_{p-1,s}, X_s)$, and let $G_i = \bigcup_s F_{i,s}$ for every $i < p$.
By property (b) of a condition, for every $i < p$, $G_i$ is $f$-thin for color $i$.
By Lemma~\ref{lem:ts-n-strongly-double-immune-assuming-force-Q}, $G_0, \dots, G_{p-1}$ are all infinite,
and by Lemma~\ref{lem:ts-n-strongly-double-immune-assuming-force-R}, there is some $i < p$ such that
$\Hcal$ is $  G_i$-\hyper.
This completes the proof of Theorem~\ref{thm:ts-n-strongly-double-immune-assuming}.
\end{proof}

\subsection{$\fs^2$ preserves \propertyL}\label{subsect:fs2-preserves}

The purpose of this section is to prove the following theorem.

\begin{theorem}\label{thm:ts-1-strongly-preserves-double}
$\fs^2$ preserves \propertyL.
\end{theorem}

We start with a lemma very similar to Lemma~\ref{lem:ts-n-to-np1-strongly-double-immune},
which establishes a bridge between strong preservation for a principle over $n$-tuples
and preservation for the same principle over $(n+1)$-tuples.

\begin{lemma}\label{lem:fs-n-to-np1-strongly-double-immune}
Fix some $n \geq 1$.
If $\fs^n$ strongly preserves \propertyL, then $\fs^{n+1}$ preserves \propertyL.
\end{lemma}
\begin{proof}
Let $\Hcal$ be a  \hyper\ family, and $f : [\omega]^{n+1} \to \omega$ be a  computable
instance of $\fs^{n+1}$. Let $\vec{R} = \langle R_{\sigma,i} : \sigma \in [\omega]^n, i \in \omega \rangle$
be the  computable family of sets defined for every $\sigma \in [\omega]^n$ and $i \in \omega$ by
$$
R_{\sigma,i} = \{ x \in \omega : f(\sigma,x) = i \}.
$$
Since $\coh$ preserves \propertyL\ (Corollary~\ref{cor:coh-preserves-double-immunity}), there
is an $\vec{R}$-cohesive set~$G$
\footnote{Here $G$ is $\v R$-cohesive iff for every $\sigma\in [\omega]^n$:
either $\lim_{x\in G}f(\sigma,x)$ exists, or
$\{f(\sigma,x) = i:x\in G\}$ is finite for all $i\in\omega$.
We are not using exactly Corollary~\ref{cor:coh-preserves-double-immunity}, but a similar
proof applies here.} such that $\Hcal$ is $ G$-\hyper.
Let $g : [\omega]^n \to \omega$ be the instance of $\fs^n$ defined for every $\sigma \in [\omega]^n$ by
$$
g(\sigma) = \cond{
	\lim_{x \in G} f(\sigma, x) & \mbox{ if it exists }\\
	0 & \mbox{ otherwise } }
$$
By strong preservation of $\fs^n$, there is an infinite $g$-free set $H\subseteq G$
such that $\Hcal$ is $ G \oplus H$-\hyper.
By thinning out the set $H$, we obtain an infinite $G \oplus H$-computable $f$-free set $\tilde{H}\subseteq H$.
In particular, $\Hcal$ is $\tilde{H}$-\hyper.
\end{proof}

We shall define a particular kind of function called \emph{left trapped function}.
The notion of trapped function
was introduced by Wang in~\cite{Wang2014Some} to prove that~$\fs$ does not imply~$\aca$
over~$\omega$-models. It was later reused by the second author in~\cite{PateyCombinatorial,Patey2016weakness}.

\begin{definition}
A function $f : [\omega]^n \to \omega$ is \emph{left (resp. right) trapped}
if for every $\sigma \in [\omega]^n$, $f(\sigma) \leq \max\sigma$ (resp. $f(\sigma) >\max\sigma$).
\end{definition}

The following lemma is a particular case of a more general statement
proven by the second author in~\cite{PateyCombinatorial}.
It follows from the facts that $\fs^n$ for right trapped functions
is strongly computably reducible to the diagonally non-computable principle ($\dnr$),
which itself is strongly computably reducible to~$\fs^n$ for left trapped functions.

\begin{lemma}[Patey in~\cite{PateyCombinatorial}]\label{lem:avoidance-fs-trapped-to-untrapped}
For each $n \geq 1$, if $\fs^n$ for left trapped functions (strongly) preserves \hyper, then so does $\fs^n$.
\end{lemma}

It therefore suffices to prove strong preservation
of \propertyL\ for $\fs^1$ restricted to left trapped functions.
We first prove a technical lemma thinning out colors while preserving
\propertyL.
Fix a  set $C$, an infinite set $X\leq_T C$,
a $C$-\hyper\ bifamily $\Hcal$
and a left trapped coloring $f : [\omega]^n \to \omega$.
\begin{lemma}\label{lem:fs-cohesiveness-strong-preservation}
Assume $\fs^s$ strongly preserves \propertyL\ for each
  $0\leq s<n$.
There exists an infinite set  $Y\subseteq X$ such that $\Hcal$ is $Y \oplus C$-\hyper,
and for every $\sigma \in [Y]^{<\omega}$ such that $0 \leq \card{\sigma} < n$,
 $$
(\forall x \in Y \setminus \sigma)(\exists b)(\forall \tau \in [Y \cap (b, +\infty)]^{n-|\sigma|})
  f(\sigma, \tau) \neq x.
$$
\end{lemma}
\begin{proof}
 By Proposition~\ref{thm:gen-coh-double-immmunity}
 and since
 $\ts^s_{d_s+1}$ strongly preserves \propertyL\ for all  $s\in\omega$
 (Theorem \ref{thm:ts-n-strongly-preserves-double}), there exists a set
$X_0\subseteq X $ with $\mcal{H}$ being $  X_0$-\hyper\
 such that for all $\sigma\in [\omega]^{<\omega}$ with $|\sigma|<n$, there exists $I_\sigma$ with
$|I_\sigma|\leq d_{n-|\sigma|}$ such that for every $x\notin I_\sigma$,
$$
(\exists b)(\forall \tau\in [X_0\cap (b,\infty)]^{n-|\sigma|})
f(\sigma,\tau)\ne x.
$$

For each $s < n$ and $i < d_{n-s}$, let $f_{s,i} : [\omega]^s \to \omega$ be the coloring
such that $f_{s,i}(\sigma)$ is the  $i$th element of $I_\sigma$.
By strong preservation of $\fs^s$   for each
  $0\leq s<n$,
  there is an infinite
  set $Y\subseteq X_0$ such that $Y$ is $f_{s,i}$-free for all $0\leq s<n,i<d_{n-s}$
  and $\Hcal$ is $Y$-\hyper.

  We prove that $Y$ is the desired set.
   Fix $s < n$,   $\sigma \in [Y]^s$
   and $x \in Y \setminus \sigma$.
If $(\forall b)(\exists \tau \in [Y \cap (b, +\infty)]^{n - s})f(\sigma, \tau) = x$,
then by choice of $X_0$ (and $Y\subseteq X_0$), there exists an $i < d_{n - s}$ such that
$f_{s, i}(\sigma) = x$, contradicting $f_{s,i}$-freeness of $Y$.
So $(\exists b)(\forall \tau \in [Y \cap (b, +\infty)]^{n - s})f(\sigma, \tau) \neq x$.
\end{proof}

\begin{theorem}\label{thm:strong-preservation-fs1-left-trapped}
$\fs^1$ for left trapped functions strongly preserves \propertyL.
\end{theorem}
\begin{proof}
Fix some  \hyper\ bifamily $\Hcal$,
and a left trapped coloring~$f : \omega \to \omega$.
By Lemma \ref{lem:fs-cohesiveness-strong-preservation},
we assume
\begin{align}\label{fstseq6}
(\forall x \in\omega)
  (\exists b)(\forall y \in  (b, +\infty))f(y) \neq x.
  \end{align}

We will construct an infinite $f$-free set~$G$ such that $\Hcal$ is $G  $-\hyper\ by forcing.
Our forcing \emph{conditions} are Mathias conditions $(F, X)$ such that
\begin{itemize}
  \item[(a)] $\Hcal$ is $X $-\hyper
  \item[(b)] $(\forall x \in F \cup X) f(x) \not \in F \setminus \{x\}$.
\end{itemize}
A Mathias condition $(F,X)$ is a \emph{precondition} if
it satisfies (a) and $F$ is $f$-free.
Clearly  $(\emptyset, \omega)$ is a condition.
 It's easy to see that:
 \begin{lemma}\label{lem:fs1-extension}
 \
 \begin{enumerate}
 \item Every precondition can be extended to a condition.
 \item For every condition  $(F,X)$, every finite $f$-free set $E\subseteq X$ with $E>F$,
 $(F\cup E, X)$ is a precondition.

 \end{enumerate}

 \end{lemma}
\begin{proof}
Item (1) follows from (\ref{fstseq6}).
For item (2): Since $f$ is left trapped,
for every $x\in F$, $f(x)\notin E$.
Combined with $f$-freeness of $F$,  $f(x)\notin (F\cup E)\setminus \{x\}$.
Since $E$ is $f$-free, for every $x\in E$,
$f(x)\notin E\setminus\{x\}$.
Combined with  property (b) of a condition,
$f(x)\notin (F\cup E)\setminus \{x\}$.

\end{proof}
\begin{lemma}\label{lem:fs1-left-trapped-preserves-1}
For every condition $(F, X)$ there exists an extension
$(E, Y)$ such that $|E| > |F|$.
\end{lemma}
\begin{proof}
Pick any $x \in X$ with $x>F$.
Clearly $\{x\}$ is $f$-free.
Thus the conclusion follow from Lemma
 \ref{lem:fs1-extension}.
\end{proof}

For every $e \in \omega$, we want to satisfy the requirement
\begin{align}\nonumber
\Rcal_e: \text{If $\Phi^{  G}_e$ is total,
then $\Phi^{  G}_e$ meets $\Hcal$.}
\end{align}

\begin{lemma}\label{lem:fs1-left-trapped-preserves-2}
For every condition $c$ and every index $e$,
there is an extension $d$ of $c$ forcing $\Rcal_e$.
\end{lemma}
\begin{proof}

Fix a condition $c = (F, X)$. By Lemma \ref{lem:fs1-extension},
for notational convenience, assume
$F=\emptyset, X=\omega$.
 We define a partial computable biarray as follows.
\bigskip

\emph{Defining $U_n$}.
Given $n \in \omega$, search  computably for some
finite   set $U_n>n$ (if it exists) such that
for every pair of left trapped colorings $g, h : \omega \to \omega$,
there is a pair of disjoint finite sets $E_0, E_1  $ which are both $g$-free and $h$-free such that
for each $i < 2$,
$$
\Phi_e^{ E_i }(n;1) \downarrow \subseteq U_n
$$

\emph{Defining $V_{n,m}$}. Given $n, m \in \omega$, search  computably for some
finite  set $V_{n,m}>m$ (if it exists) such that for every
left trapped coloring $g : \omega \to \omega$,
there is some $g$-free finite set $E$ such that
$$
\Phi_e^{ E}(n;1) \downarrow \subseteq U_n
	\wedge \Phi_e^{ E}(n, m;2) \downarrow \subseteq V_{n,m}
$$
\smallskip

We now have multiple outcomes, depending on which $U_n$ and $V_{n,m}$ is found.
\begin{itemize}
	\item Case 1: $U_n$ is not found for some $n \in \omega$. By compactness, the following $\Pi^{0}_1$ class
	$\Pcal$ of  pairs of left trapped colorings $g, h : \omega \to \omega$ is nonempty:
	 for every pair of pairwise disjoint finite sets $E_0, E_1$
	which are both $g$-free and $h$-free, there is some $i < 2$ such that
	$\Phi_e^{ E_i}(n;1) \uparrow$.

	As $\wkl$ preserves \propertyL\
	(Corollary~\ref{cor:wkl-double-immunity}), there is a member $g, h$
	of $\Pcal$ such that
	$\Hcal$ is $g \oplus h $-\hyper.
	There is an infinite $g\oplus h$-computable set $Y$ which is both
	$g$-free and $h$-free.
	Let $E_0 $ (if it exists) be   $g$-free, $h$-free
and $\Phi_e^{E_0}(n;1) \downarrow$; let $b = \max E_0$ (or $b=0$ if $E_0$ does not exist).
	Clearly condition  $d = (F, Y \setminus [0,b])$ is an extension of $c$.
For every $G$ satisfying $d$, $G$ is both $g$-free and $h$-free,
so $\Phi_e^G(n;1)\uparrow$. Thus $d$ forces $\Rcal_e$.
	\bigskip

	\item Case 2: $U_n$ is found, but not $V_{n,m}$ for some $n, m \in \omega$. By compactness,
	the following $\Pi^{0}_1$ class $\Pcal$ of left trapped colorings $g : \omega \to \omega$
is nonempty:
	 for every $g$-free set $E$,
	\begin{align}\label{fsts-fs1strongpreserve-defofg}
	\Phi_e^{ E}(n;1) \downarrow \subseteq U_n
		\Rightarrow \Phi_e^{ E}(n, m;2) \uparrow.
	\end{align}
	As $\wkl$ preserves \propertyL\
	(Corollary~\ref{cor:wkl-double-immunity}), there is a member $g$
	of $\Pcal$ such that $\Hcal$ is $g$-\hyper.
	There is an infinite  $g$-computable
 $g$-free set $Y$.
	By definition of $U_n$ (where we take $h = f$),
	there is a pair of disjoint finite sets $E_0, E_1 $
	which are both $g$-free and $f$-free, and such that for every $i < 2$,
	$$
	\Phi_e^{E_i}(n;1) \downarrow \subseteq U_n.
	$$
Consider the 2-partition $A_0\sqcup A_1$ of $\omega$
defined by $x\in A_i$ if $g(x)\in E_i$ and $x\in A_0$ if $g(x)\notin \cup_{i<2}E_i$.	
 Since  $\ts^1_2$ is computably true, hence preserves \propertyL,
	there is an $i < 2$ and an infinite set $E_i<\t Y \subseteq Y$ such that $\Hcal$
	is $\t Y $-\hyper\ and $g(\t Y)\cap E_i=\emptyset$.
We claim that $E_i \cup \t Y$ is $g$-free.
	Indeed, $E_i$ and $\t Y$ are both $g$-free; since $g$ is left trapped,
	$g(E_i) \cap \t Y = \emptyset$; and by choice of $\t Y$,
$g(\t Y) \cap E_i = \emptyset$.

		By $f$-free of $E_i$, $d=(E_i, \t Y  )$ is a precondition extending $c$.
We prove that the   $d$ forces $\Phi_e^{  G}(n,m;2) \uparrow$.
Because   $\Phi_e^{E_i}(n;1)\downarrow\subseteq U_n$ and
$E_i \cup \t Y$ is $g$-free
(so every $G$ satisfying $(E_i,\t Y)$ is $g$-free), by definition of $g$ (namely (\ref{fsts-fs1strongpreserve-defofg})), $d$ forces $\Phi_e^{ G}(n, m;2) \uparrow$.

	\bigskip

	\item Case 3: $U_n$ and $V_{n,m}$ are found for every $n,m \in \omega$.
	By  \propertyL\ of $\Hcal$, there is some $n, m \in \omega$
	such that $(U_n, V_{n,m}) \in \Hcal$.
	In particular, by definition of $V_{n,m}$ (where we take $g = f$),
	there is some $f$-free finite set $E $ such that
	$$
	\Phi_e^{E}(n;1) \downarrow \subseteq U_n
		\wedge \Phi_e^{E}(n, m;2) \downarrow \subseteq V_{n,m}.
	$$
	
		Clearly $(E, X)$ is a  precondition
extending $c$ and forcing $\Rcal_e$.
\end{itemize}
This completes the proof of Lemma~\ref{lem:fs1-left-trapped-preserves-2}.
\end{proof}

Let $\Fcal = \{c_0, c_1, \dots \}$ be a sufficiently generic filter for this notion of forcing,
where $c_s = (F_s, X_s)$, and let $G = \bigcup_s F_s$.
By property (b) of a condition, $G$ is $f$-free.
By Lemma~\ref{lem:fs1-left-trapped-preserves-1}, $G$ is infinite,
and by Lemma~\ref{lem:fs1-left-trapped-preserves-2},
$\Hcal$ is $  G$-\hyper.
This completes the proof of Theorem~\ref{thm:strong-preservation-fs1-left-trapped}.
\end{proof}

\subsection{
 $\fs^n$ preserves \propertyL
}\label{subsect:fsn-preserves}

The purpose of this section is to prove the following theorem.

\begin{theorem}\label{thm:fs-n-strongly-preserves-double}
For every $n \geq 1$, $\fs^n$ strongly preserves \propertyL.
\end{theorem}
\begin{proof}
We prove by induction over $n \geq 1$ that
$\fs^n$ preserves and strongly preserves \propertyL.
If $n = 1$, $\fs^1$ is a computably true statement, that is, every instance
has a solution computable in the instance, so $\fs^1$ preserves \propertyL.
If $n > 1$, then by the induction hypothesis, $\fs^{n-1}$ strongly preserves \propertyL,
so by Lemma~\ref{lem:fs-n-to-np1-strongly-double-immune}, $\fs^n$ preserves \propertyL.
Assuming by the induction hypothesis that $\fs^t$ strongly preserves \propertyL\
for every $t < n$, and that $\fs^n$ preserves \propertyL,
then by Theorem~\ref{thm:strong-preservation-fs-left-trapped},
$\fs^n$ for left trapped functions strongly preserves \propertyL.
By Lemma~\ref{lem:avoidance-fs-trapped-to-untrapped}, if $\fs^n$ for left trapped functions strongly preserves \propertyL,
so does $\fs^n$. This completes the proof.
\end{proof}

We first need to prove a technical lemma which will ensure that the reservoirs of our forcing conditions
will have good properties, so that the conditions will be extensible.
Fix a set $C$, a $C$-\hyper\ bifamily $\Hcal$; a finite set $F$ and an infinite set $X \leq_T C$;
fix a left trapped coloring $f : [\omega]^n \to \omega$.
\begin{lemma}\label{lem:fs-left-trapped-preserves-smaller}
Suppose that $\fs^s$ strongly preserves \propertyL\ for each $0<s<n$.
Then there exists an infinite
set $Y \subseteq X$ such that $\Hcal$ is $Y \oplus C$-\hyper, and for each $0<s<n$,
$$
(\forall \sigma \in [F]^s)(\forall \tau \in [Y]^{n-s})f(\sigma, \tau) \not \in Y \setminus \tau.
$$
\end{lemma}
\begin{proof}

For each $0<s<n$, each $\sigma\in [F]^s$ consider the coloring
$f_\sigma:[\omega]^{n-s}\rightarrow\omega$ defined as $f_\sigma(\tau) = f(\sigma,\tau)$.
Since  $\fs^s$ strongly preserves \propertyL\ for each $0<s<n$,
there is an infinite
set $Y \subseteq X$ such that $\Hcal$ is $Y$-\hyper\
and $Y$ is $f_\sigma$-free for all $0<s<n,\sigma\in [F]^s$.
Unfolding the definition of free, $Y$ is desired.
\end{proof}

We need to prove a second lemma saying that if we have sufficiently many finite free sets,
one of them is extendible into an infinite one.
This generalizes the argument of case 2 of Lemma
\ref{lem:fs1-left-trapped-preserves-2}, where we showed that
for every coloring $g:\omega\rightarrow \omega$ and every pair of disjoint $g$-free finite sets $E_0, E_1$,
 there is an $i<2$ and an infinite set $Y$ such that $E_i\cup Y$ is $g$-free.
\begin{definition}[\fssufficient]
\label{def:fssuff}
A collection $\mcal{F}$ of finite sets is \emph{$n$-\fssufficient} iff for every sequence $(f_{s,i}:[\omega]^s\rightarrow\omega)_{s< n,i<d_{n-s-1}}$
of left trapped colorings, there exists an $F\in\mcal{F}$ such that
$F$ is $f_{s,i}$-free for all $s< n,i<d_{n-s-1}$.
\end{definition}

In particular, for $n = 1$, a collection $\mcal{F}$ of finite sets is 1-\fssufficient\ iff for every left-trapped coloring $f_{0,0} : [\omega]^0 \to \omega$, there exists an $F \in \mcal{F}$ such that $F$ is $f_{0,0}$-free. Note that a coloring $[\omega]^\omega \to \omega$ is nothing but the choice of an element in $\omega$, which means that for every $x \in \omega$, there exists an $F \in \mcal{F}$ such that $x \not \in F$. This is true as long as $\mcal{F}$ contains two disjoint sets.

Let $\mcal{F}$ be a collection of $f$-free finite sets.
\begin{lemma}\label{lem:fs-with-enough-sets-free}
Assume that $\fs^n$ preserves \propertyL\
and that $\fs^s$ strongly preserves \propertyL\ for each $0<s<n$.
Suppose $f\leq_T C$, $\mcal{F}$ is $n$-\fssufficient\ and   $f$-free\footnote{Each member of $\mcal F$ is $f$-free.}.
Then
there is an $F\in\mcal{F}$
and an infinite subset $Y \subseteq X$
such that $F \cup Y$ is $f$-free and $\Hcal$ is $C \oplus Y$-\hyper.
\end{lemma}
\begin{proof}
Using a version of Proposition \ref{thm:gen-coh-double-immmunity}
and imitating Lemma \ref{lem:fs-cohesiveness-strong-preservation} but for preservation instead of strong preservation
 (which is feasible since $f \leq_T C$),
there is an infinite set $X_0\subseteq X$
with $\Hcal$ being $X_0$-\hyper\ such that
for every $\sigma\in [\omega]^{<\omega}$ with $|\sigma|< n$,
there is a $I_\sigma$ with $|I_\sigma|\leq d_{n - |\sigma|-1}$
such that
for every $x\notin I_\sigma$,
$$
(\exists b)(\forall \tau\in [X_0\cap(b,\infty)]^{n-|\sigma|})
f(\sigma,\tau)\ne x.
$$
For each $s< n, i<d_{n-s-1}$,
consider the coloring $f_{s,i}:[\omega]^s\rightarrow \omega$
such that $f_{s,i}(\sigma)$ is the $i$th element of $I_\sigma$.

Since $\mcal{F}$ is $n$-\fssufficient,
there is an $F\in\mcal{F}$
such that $F$ is $f_{s,i}$-free for all $s< n,i<d_{n-s-1}$.
By choice of $X_0$, let $b\in\omega$ so that
for every $\sigma\in [F]^{<\omega}$ with $|\sigma|<n$,
\begin{align}\label{fstseq9}
(\forall \tau\in [X_0\cap (b,\infty)]^{n-|\sigma|})
f(\sigma,\tau)\notin F\setminus I_\sigma.
\end{align}
By preservation of \hyper\ of $\fs^n$,
let infinite $f$-free set $X_1\subseteq X_0\cap (b,\infty)$
so that $\Hcal$ is $X_1$-\hyper.
By Lemma \ref{lem:fs-left-trapped-preserves-smaller},
let infinite set $Y\subseteq X_1$
so that $\Hcal$ is $Y$-\hyper, $Y>F$
and for every $0<s<n$,
\begin{align}\label{fstseq8}
(\forall \sigma \in [F]^s)(\forall \tau \in [Y]^{n-s})f(\sigma, \tau) \not \in Y \setminus \tau.
\end{align}

We show that $F\cup Y$ is $f$-free.
To see this, let $\sigma\in [F]^{<\omega},\tau\in [Y]^{<\omega}$
with $|\sigma\cup \tau|=n$, we need to prove
$f(\sigma,\tau)\notin F\setminus\sigma$
and $f(\sigma,\tau)\notin Y\setminus \tau$.
To see $f(\sigma,\tau)\notin Y\setminus \tau$,
when $|\tau| = n$, the conclusion follows by $f$-free of $Y$;
when $0<|\tau|<n$, the conclusion follows from (\ref{fstseq8});
when $|\tau| = 0$, the conclusion follows by left trap of $f$.
To see $f(\sigma,\tau)\notin F\setminus\sigma$,
when $|\sigma| = n$, the conclusion follows from $f$-freeness of $F$.
When $|\sigma|<n$,
suppose $f(\sigma,\tau)\in F$ (otherwise we are done).
By (\ref{fstseq9}),
 $f(\sigma,\tau)=x\in F\cap I_\sigma$.
Since $x\in I_\sigma$, it means for  $s=n-|\sigma|$,
for some $i<d_{n-s-1}$, $f_{s,i}(\sigma)=x$.
In particular, $f_{s,i}(\sigma)\in F$.
But $F$ is $f_{s,i}$-free, so $x\notin \sigma$. Thus we are done.

\end{proof}

We are now ready to prove the missing theorem.

\begin{theorem}\label{thm:strong-preservation-fs-left-trapped}
For each~$n \geq 1$, if~$\fs^s$ strongly preserves \propertyL\
for each~   $0\leq s<n$  and $\fs^n$ preserves \propertyL,
then $\fs^n$ for left trapped functions strongly preserves \propertyL.
\end{theorem}
\begin{proof} 
Fix  a left trapped coloring~$f : [\omega]^n \to \omega$.
By Lemma~\ref{lem:fs-cohesiveness-strong-preservation},
we assume
that
for every $s<n$, every $\sigma\in [\omega]^s$,
\begin{align}\label{fstseq7}
(\forall x \in \omega \setminus \sigma)
  (\exists b)(\forall \tau \in [\omega\cap (b, +\infty)]^{n-s})
  f(\sigma, \tau) \neq x.
\end{align}

We will construct an infinite $f$-free set~$G$ such that $\Hcal$ is $G  $-\hyper.
Our forcing \emph{conditions} are Mathias conditions $(F, X)$ such that
\begin{itemize}
  \item[(a)] $\Hcal$ is $X  $-\hyper.
  \item[(b)] $(\forall \sigma \in [F \cup X]^n) f(\sigma) \not \in F \setminus \sigma$.
  \item[(c)] $(\forall \sigma \in [F]^s)(\forall \tau \in [X]^{n-s})
  f(\sigma, \tau) \not \in X \setminus \tau$ for each   $0<s<n$.
\end{itemize}
Clearly $(\emptyset,\omega)$ is a condition.
A \emph{precondition} $(F,X)$ is a Mathias condition satisfying (a) and 
where $F$ is
$f$-free.
\begin{lemma}\label{lem:fsn-extension}
\

\begin{enumerate}
\item Every precondition can be extended to a condition.
\item For every condition  $(F,X)$, every $f$-free set $Y\subseteq X$
with $Y>F$,
$F\cup Y$ is $f$-free.
\item  For every condition  $(F,X)$, every finite $f$-free set $E\subseteq X$ with $E>F$,
 $(F\cup E, X )$ is a precondition.
\end{enumerate}
\end{lemma}
\begin{proof}
For item (1):
Fix a precondition $(F,X)$.
 By (\ref{fstseq7}) and $f$-freeness of $F$, there is a $b\in\omega$
 so that  for every $\sigma\in [F]^{\leq n}$,
 $$
 (\forall \tau \in [\omega\cap (b, +\infty)]^{n-|\sigma|})
  f(\sigma, \tau) \notin F\setminus \sigma,
  $$
which verifies that $(F,X\cap (b,\infty))$ satisfies property (b).
By Lemma \ref{lem:fs-left-trapped-preserves-smaller},
there is
a reservoir-extension\footnote{A Mathias condition $(E,Y)$
reservoir-extends $(F,X)$ if it extends $(F,X)$ and $E=F$.} of  $(F,X\cap (b,\infty))$ satisfying property (c) while preserving property (a).
Thus we are done (property (b) is preserved by reservoir-extension).

For item (2):
Let $\sigma\in [F]^{<\omega}, \tau\in [Y]^{<\omega}$
with $|\sigma\cup \tau| = n$. We need to show
that $f(\sigma, \tau)\notin F\setminus \sigma$ and $
f(\sigma,\tau)\notin Y\setminus \tau$.
It follows from property (b)
of $(F,X)$ that $f(\sigma,\tau)\notin F\setminus\sigma$.
To see $f(\sigma,\tau)\notin Y\setminus \tau$,
the conclusion follows from property (c) of $(F,X)$
when $|\sigma|,|\tau|>0$;
the conclusion follows by left trap of $f$ when
$|\tau| = 0$; the conclusion follows from $f$-freeness of $Y$
when $|\tau| = n$.

Item (3) follows from item (2) directly.
\end{proof}

\begin{lemma}\label{lem:fs2-left-trapped-preserves-1}
For every condition $(F, X)$ there exists an extension
$(E, Y)$ such that $|E| > |F|$.
\end{lemma}
\begin{proof}
Pick any $x \in X$ so that $x>E$ and set $E = F \cup \{x\}$.
Since $\{x\}$ is $f$-free, so by Lemma \ref{lem:fsn-extension}, $(E,X)$
is a precondition.
\end{proof}

For every $e \in \omega$, we want to satisfy the requirement
\begin{align}\nonumber
\Rcal_e: \text{If $\Phi^{  G}_e$ is a total,
then $\Phi^{ G}_e$ meets $\Hcal$.}
\end{align}

\begin{lemma}\label{lem:fs2-left-trapped-preserves-2}
For every condition $c$ and every index $e$,
there is an extension $d$ of $c$ forcing $\Rcal_e$.
\end{lemma}
\begin{proof}
Fix $c = (F, X)$.
By Lemma \ref{lem:fsn-extension},
for notational convenience, assume $F=\emptyset$ and $X=\omega$.
 We define a partial computable biarray as follows.

\bigskip

\emph{Defining $U_n$}.
Given $r \in \omega$, search  computably for some
finite  set $U_r>r$ (if it exists) such that
for every pair of left trapped colorings $g, h : [\omega]^n \to \omega$,
there is a finite $n$-\fssufficient\ collection
$\mcal{E}$ of finite sets which is both $g$-free and $h$-free such that for every $E\in\mcal{E}$,
$$
\Phi_e^{E }(r;1) \downarrow \subseteq U_r.
$$

\emph{Defining $V_{r,m}$}. Given $r, m \in \omega$, search  computably for some
finite  set $V_{r,m}>m$ (if it exists) such that for every
left trapped coloring $g : [\omega]^n \to \omega$,
there is some $g$-free finite set $E$ such that
$$
\Phi_e^{E}(r;1) \downarrow \subseteq U_r
	\wedge \Phi_e^{E}(r, m;2) \downarrow \subseteq V_{r,m}.
$$
\smallskip

We now have multiple outcomes, depending on which $U_r$ and $V_{r,m}$ is found.
\begin{itemize}
	\item Case 1: $U_r$ is not found for some $r \in \omega$. By compactness, the
following $\Pi^{0}_1$ class
	$\Pcal$ of pairs of left trapped colorings $g, h : [\omega]^n \to \omega$
is nonempty:
	there is no   $n$-\fssufficient\ finite collection
$\mcal{E}$ of finite sets which are both $g$-free and $h$-free,
	 such that for every $E\in\mcal{E}$, we have
	$\Phi_e^{E}(r;1) \downarrow$.

	As $\wkl$ preserves \propertyL\
	(Corollary~\ref{cor:wkl-double-immunity}), there is a member $g, h $
of $\Pcal$ such that
	$\Hcal$ is $g \oplus h$-\hyper.
Unfolding the definition of $n$-\fssufficient\ and use compactness,
the following  $\Pi_1^{0,g\oplus h}$ class $\mcal{Q}$
of sequence $(f_{s,i}:[\omega]^s\rightarrow\omega)_{s<n,i<d_{n-s-1}}$ of left trapped colorings
is nonempty:
\begin{align}\label{fstseq10}
&\text{for every finite set $E$ which is $g$-free, $h$-free and }\\ \nonumber
&\text{$f_{s,i}$-free for each $s< n,i<d_{n-s-1}$,
$\Phi_e^{E}(r;1) \uparrow$.}
\end{align}
 As $\wkl$ preserves \propertyL\
	(Corollary~\ref{cor:wkl-double-immunity}),
there is a member $(f_{s,i}:s< n,i<d_{n-s-1})$ of $\mcal{Q}$
such that $\Hcal$ is $g\oplus h\oplus_{s< n,i<d_{n-s-1}}f_{s,i}$-\hyper.
	As $\fs^n$ preserves \propertyL, there is an infinite set $Y $ which is both
	$g$-free, $h$-free and $f_{s,i}$-free for each $s< n,i<d_{n-s-1}$ and such that $\Hcal$ is $Y $-\hyper.
	Clearly for every $G$ satisfying condition $(F,Y)$,
$G$ is $g$-free, $h$-free and $f_{s,i}$-free for each $s< n,i<d_{n-s-1}$,
so 	$\Phi_e^{  G}(r;1) \uparrow$.
i.e.,
	The condition $d = (F, Y)$ is an extension of $c$ forcing $\mcal{R}_e$.
	\bigskip

	\item Case 2: $U_r$ is found, but not $V_{r,m}$ for some $r, m \in \omega$. By compactness,
	the following $\Pi^{0}_1$ class $\Pcal$ of  left trapped colorings $g : [\omega]^n \to \omega$
is nonempty:
	 for every $g$-free set $E $,
	\begin{align}\label{fsts-fsnstongpreserve-defofg}
	 \Phi_e^{E}(r;1) \downarrow \subseteq U_r\Rightarrow
		 \Phi_e^{E}(r, m;2) \uparrow.
	\end{align}
  As $\wkl$ preserves \propertyL\
	(Corollary~\ref{cor:wkl-double-immunity}), there is a member $g$
of $\Pcal$ such that $\Hcal$ is $g  $-\hyper.
	By definition of $U_r$ (where we take letting $h = f$),
	there is a $n$-\fssufficient\ finite collection $\mcal{E}$
of finite sets
	which is both $g$-free and $f$-free and such that for each $E\in\mcal{E}$,
	$$
	\Phi_e^{E}(r;1) \downarrow \subseteq U_r.
	$$

	By Lemma~\ref{lem:fs-with-enough-sets-free}, there is an  infinite set $Y$
	and some $E\in\mcal{E}$ such that $\Hcal$ is $Y $-\hyper\ and $E \cup Y$
	is $g$-free. Consider the precondition $d=(E, Y)$.
It remains to prove that  $d$ forces $\Phi_e^{ G}(r, m;2) \uparrow$.
	Since $E \cup Y$ is $g$-free, so every $G$ satisfying $(E,Y)$ is $g$-free.
By definition of $g$ (namely (\ref{fsts-fsnstongpreserve-defofg})) and $\Phi_e^{E}(r;1)\downarrow\subseteq U_r $,
for every $G$ satisfying $(E,Y)$,
$\Phi_e^{ G}(r, m;2) \uparrow$.

	\bigskip

	\item Case 3: $U_r$ and $V_{r,m}$ are found for every $r,m \in \omega$.
	By  \propertyL\ of $\Hcal$, there is some $r, m \in \omega$
	such that $(U_r, V_{r,m}) \in \Hcal$.
	In particular, by definition of $V_{r,m}$ (where we take   $g = f$),
	there is some $f$-free finite set $E$ such that
	$$
	\Phi_e^{E}(r;1) \downarrow \subseteq U_r
		\wedge \Phi_e^{E}(r, m;2) \downarrow \subseteq V_{r,m}.
	$$
Consider the precondition $(E, X)$.
Clearly it forces $\Rcal_e$.
\end{itemize}
This completes the proof of Lemma~\ref{lem:fs2-left-trapped-preserves-2}.
\end{proof}

Let $\Fcal = \{c_0, c_1, \dots \}$ be a sufficiently generic filter for this notion of forcing,
where $c_s = (F_s, X_s)$, and let $G = \bigcup_s F_s$.
By property (b) of a condition, $G$ is $f$-free.
By Lemma~\ref{lem:fs2-left-trapped-preserves-1}, $G$ is infinite,
and by Lemma~\ref{lem:fs2-left-trapped-preserves-2},
$\Hcal$ is $C \oplus G$-\hyper.
This completes the proof of Theorem~\ref{thm:strong-preservation-fs-left-trapped}.
\end{proof}




\section{Erd\H{o}s-Moser theorem has no universal instance}
\label{uemsec1}

In this section,
we prove Theorem \ref{uemth0},
that $\msf{EM}$ does not have a universal instance.
To this end, we construct a pair of computable $\msf{EM}$ instances
$T_0,T_1$ such that for every computable $\msf{EM}$ instance $T$,
$T$ admits a solution that either computes no
solution of $T_0$ or computes no  solution of $T_1$.
Given an   $\msf{EM}$ instance
$T$ and two sets $A,B$,
we write $A\r_T B$ iff
for every $x\in A, y\in B$, $T(x,y)$
\footnote{It is helpful to picture $T$ as a directed graph.};
we say $T$ \emph{\diagonalagainst}
$(A,B)$
if: $A\r_{T} B$ and for all but finitely many
$x\in\omega$, $B\r_{T} x\r_{T} A$.
The point is, when $T$  \diagonalagainst\ $(A,B)$,
for any set $H$ that has nonempty intersection with
both $A,B$,  there is no solution to $T$ containing
$H$.

\begin{definition}
\
\begin{enumerate}
\item A \emph{4-\iatfs}
is a sequence of $4$-tuple
 of finite sets (of integers) $\langle E_n,E_{n,m,l}, F_{n,m},
 F_{n,m,l}: n,m,l\in\omega\rangle$
such that for every $n,m,l\in \omega$, $E_n>n,
E_{n,m,l}>  m, F_{n,m}>n$
and $F_{n,m,l}> l$.

\item A pair of $\msf{EM}$ instances
$(T_0,T_1) $ is $C$-\emph{\iatfshyp}
if for every $C$-computable
$4$-\iatfs\  $\langle E_n,E_{n,m,l}, F_{n,m},
 F_{n,m,l}: n,m,l\in\omega\rangle$,
there exist $n,m,l\in\omega$  such that
$T_0$ \diagonalagainst\ $(E_{n},E_{n,m,l} )$
and $T_1$ \diagonalagainst\ $(F_{n,m},F_{n,m,l})$.

\end{enumerate}
\end{definition}

For notational convenience,
in this section we regard each Turing machine $\Phi$ as computing a 4-\iatfs. We will therefore assume that whenever $\Phi(n;1)$ converges,
then it will output (the canonical index of) a finite  set $E_n > n$. Similarly for $\Phi(n,m, l;2)$, $\Phi(n,m;3)$ and $\Phi(n,m;4)$ with the appropriate lower bound.

By finite injury argument (as
Proposition \ref{thm:sts23-priority}), we have:
\begin{proposition}\label{uemprop1}
There exists a pair of computable stable \iatfshyp\ $\msf{EM}$
instance.
\end{proposition}
\begin{proof}
We build the tournaments $T_0$ and $T_1$ by a finite injury priority argument. For simplicity, we see $T_0$ and $T_1$ as functions over $f_0, f_1 : [\omega]^2 \to 2$ by letting for every $x < y$ and $i < 2$, $T_i(x, y)$ hold iff $f_i(x, y) = 0$.
For every $e \in \omega$, we want to satisfy the following requirement:
\begin{align}
\Rcal_e: &\text{If $\Phi_e$ is total, then there is some $n, m, l \in \omega$
such that}\\ \nonumber
&\text{$T_0$ diagonalizes against $(\Phi_e(n;1), \Phi_e(n, m, l;2))$ and }\\ \nonumber
& \text{$T_1$ diagonalizes against $(\Phi_e(n,m;3), \Phi_e(n, m, l;4))$.}
\end{align}
The requirements are given the usual priority ordering $\Rcal_0 < \Rcal_1 < \dots$
Initially, the requirements are neither partially, nor fully satisfied.
\begin{itemize}
	\item[(i)] A requirement $\Rcal_e$ \emph{requires a first attention at stage~$s$}
	if it is not first satisfied and $\Phi_{e,s}(n;1) \downarrow = E_n$ for some set $E_n \subseteq \{e+1, \dots, s-1\}$
	such that no element in $E_n$ is restrained by a requirement of higher priority.
	If it receives attention, then it puts a restraint on $E_n$,
	commits the elements of $E_n$ to be in $C_0(f_0)$, and is declared \emph{first satisfied}.
	
	\item[(ii)] A requirement $\Rcal_e$ \emph{requires a second attention at stage~$s$}
	if it is not second satisfied and $\Phi_{e,s}(n;1) \downarrow = E_n$ and $\Phi_{e,s}(n,m;3) \downarrow = F_{n,m}$ for some sets $E_n, F_{n,m} \subseteq \{e+1, \dots, s-1\}$
	such that no element in $E_n \cup F_{n,m}$ is restrained by a requirement of higher priority
	and such that $f_0(x, y) = 0$ for every $x \in E_n$ and $y \in \{m+1, m+2, \dots, s-1\}$.
	If it receives attention, then it puts a restraint on $E_n \cup F_{n,m}$,
	commits the elements of $E_n$ to be in $C_0(f_0)$ and the elements of $F_{n,m}$ to be in $C_0(f_1)$. Then the requirement is declared \emph{second satisfied}.

	\item[(iii)] A requirement $\Rcal_e$ \emph{requires a third attention at stage~$s$}
	if it is not fully satisfied, and $\Phi_{e,s}(n;1) \downarrow = E_n$, $\Phi_{e,s}(n, m, l;2) \downarrow = E_{n,m,l}$, $\Phi_{e,s}(n,m;3) \downarrow = F_{n,m}$ and $\Phi_{e,s}(n,m,l;4) \downarrow = F_{n,m,l}$
	for some sets $E_n, E_{n,m,l}, F_{n,m}, F_{n,m,l} \subseteq \{e+1, \dots, s-1\}$ which are not restrained	by a requirement of higher priority, and such that $f_0(x, y) = 0$ for every $x \in E_n$ and $y \in \{m+1, m+2, \dots, s-1\}$, and $f_1(x, y) = 0$ for every $x \in F_{n,m}$ and $y \in \{l+1, l+2, \dots, s-1 \}$.
	If it receives attention,
	then it puts a restraint on $E_n \cup E_{n,m,l} \cup F_{n,m} \cup F_{n,m,l}$, commits the elements of $E_n$ to be in $C_1(f_0)$,
	the elements of $E_{n,m,l}$ to be in $C_0(f_0)$,
	the elements of $F_{n,m}$ to be in $C_1(f_1)$,
	the elements of $F_{n,m,l}$ to be in $C_0(f_1)$, and is declared \emph{fully satisfied}.
\end{itemize}
At stage~0, we let $f_0 = f_1 = \emptyset$.
Suppose that at stage $s$, we have defined $f_0(x, y)$ and $f_1(x, y)$ for every $x < y < s$.
For every $x < s$ and $i < 2$, if it is committed to be in some $C_j(f_i)$, set $f_i(x, s) = j$,
and otherwise set $f_i(x, s) = 0$.
Let $\Rcal_e$ be the requirement of highest priority which requires attention.
If $\Rcal_e$ requires a third attention,
then execute the third procedure. Otherwise, if it requires the second attention, then execute the second procedure, and in the last case, execute the first one.
In any case, reset all the requirements of lower priorities by setting them unsatisfied,
releasing all their restraints, and go to the next stage.
This completes the construction.
On easily sees by induction that each requirement acts finitely often, and is eventually
fully satisfied. This procedure also yields stable colorings, hence stable tournaments.
\end{proof}

Before proving our core argument which will be Theorem \ref{uemth1}, we prove a few preservation results. These results will be used to assume some good properties on our tournaments.

\begin{proposition}\label{uemprop0}
$\msf{COH}$ preserves \propiatfshyp. i.e.,
For every set $C$ and $C$-\iatfshyp\ $\msf{EM}$
instance pair $(T_0,T_1)$,  and every $C$-computable $\msf{COH}$
instance $\vec{R}$, there exists a solution $G$ of $\vec{R}$ such that
$(T_0,T_1)$ is $C\oplus G$-\iatfshyp.
\end{proposition}
\begin{proof}
Let $\Bcal$ be the class of all $4$-arrays such that for every $m, n, l \in \omega$, either  $T_0$ does not diagonalize against  $(E_{n},E_{n,m,l} )$ or
 $T_1$ does not diagonalize against $(F_{n,m},F_{n,m,l})$. The class $\Bcal$ can be coded as a closed set in the Baire space $\omega^\omega$. By hypothesis, $\Bcal$ has no $C$-computable member.
 By~\cite[Corollary 2.9]{PateyCombinatorial},
there is an $\vec{R}$-cohesive set~$G$ such that $\Bcal$ has no $C \oplus G$-computable member.
By definition of $\Bcal$, $(T_0,T_1)$ is $C\oplus G$-\iatfshyp.
\end{proof}

We also need the preservation of \propiatfshyp\ of $\msf{WKL}$.

\begin{proposition}\label{uemprop2}
$\msf{WKL}$ preserves \propiatfshyp. i.e.,
For every set $C$ and $C$-\iatfshyp\ $\msf{EM}$
instance pair $(T_0,T_1)$,
 every nonempty $\Pi_1^{0,C}$ class
$\mcal{P}\subseteq 2^\omega$,
there is a $G\in\mcal{P}$ such that
$(T_0,T_1)$ is $C\oplus G$-\iatfshyp.

\end{proposition}
\begin{proof}
Assume $C = \emptyset$, and fix $(T_0, T_1)$ and a $\Pi^0_1$ class $\mcal{P} \subseteq 2^\omega$.
We will prove our proposition with a forcing with $\Pi^0_1$ non-empty subclasses of $\mcal{P}$. We satisfy the requirement:
\begin{align}\nonumber
\mcal{R}_e: & \text{If }\Phi_e^G\text{ is total, then for
some $n,m,l \in \omega$, }\\ \nonumber
& T_0\text{ \diagonalagainst\ $(\Phi_e^G(n;1),\Phi_e^G(n,m,l;2))$ and}\\ \nonumber
&\text{$T_1$ \diagonalagainst\ $(\Phi_e^G(n,m;3),\Phi_e^G(n,m,l;4))$.}
\end{align}
The core of the argument is the following lemma :

\begin{lemma}\label{lem:uemprop2-preserve}
For every index $e$, every condition $c$
 admits an extension forcing $\mcal{R}_e$.
\end{lemma}
\begin{proof}
Let $\mcal{Q} \subseteq \mcal{P}$ be a condition.
We  define 
 a partial computable 4-\iatfs\ as follows.
\bigskip

\emph{Defining $U_n$}.
Given $n \in \omega$, search computably for some
finite set $U_n>n$ such that
for every $X\in \mcal{Q}$,
$$\Phi_e^X(n;1)\downarrow\subseteq U_n.$$

\emph{Defining $V_{n,m}$}.
 Given $n,m\in\omega$, search computably for some
 finite set $  V_{n,m}>n$
such that
for every $X\in \mcal{Q}$,
$$
\Phi_{e}^{X}(n, m;2) \downarrow \subseteq V_{n,m}.
$$

 \emph{Defining $U_{n,m,l},V_{n,m,l}$}.
  Given $n,m,l\in\omega$, search computably for some
 finite sets $U_{n,m,l}>m, V_{n,m,l}>l$
 such that
for every $X\in \mcal{Q}$,
\begin{align}\nonumber
& \Phi_e^X(n,m,l;3)\downarrow\subseteq U_{n,m,l}\wedge
\Phi_e^X(n,m,l;4)\downarrow\subseteq V_{n,m,l}.
\end{align}

\smallskip

We now have multiple outcomes, depending on which $U_n$ and $V_{n,m}$ is found.
\begin{itemize}
	\item Case 1: $U_n$ is not found for some $n \in \omega$.
	Then by compactness,
the  $\Pi_1^0$ class $\mcal{W}$ of $X\in \mcal{Q}$ so that
$\Phi_e^X(n;1)\uparrow$ is nonempty.
Thus $\mcal{W}$ is the desired extension.

	\item Case 2: $V_{n,m}$ is not found for some $n, m \in \omega$.
Then by compactness,
the  $\Pi_1^0$ class $\mcal{W}$ of $X\in \mcal{Q}$ so that
$\Phi_{e}^{X}(n, m;2) \uparrow$ is nonempty.
Thus $\mcal{W}$ is the desired extension.

	\item Case 3: $U_n$ or $V_{n,m}$ is not found for some $n,m \in \omega$.
	Then by compactness,
the  $\Pi_1^0$ class $\mcal{W}$ of $X\in \mcal{Q}$ so that
\begin{align}\nonumber
&\Phi_e^X(n,m,l;3)\uparrow\vee
\Phi_e^X(n,m,l;4)\uparrow
\end{align}
is nonempty.
Thus $\mcal{W}$ is the desired extension.

\item Case 4: $U_n, V_{n,m},U_{n,m,l},V_{n,m,l}$ are found for every $n,m \in \omega$.
Since $(T_0,T_1)$ is \iatfshyp, there exist
$n, m,l$ such that $T_0$,$T_1$ \diagonalagainst\
$(U_{n},U_{n,m,l})$ and
$(V_{n, m},V_{n, m,l})$ respectively.
Thus $\mcal{Q}$ already forces $\mathcal{R}_e$.
\end{itemize}

\end{proof}

Let $\Fcal = \{\mcal{P}_0, \mcal{P}_1, \dots \}$ be a sufficiently generic filter for this notion of forcing,
where $c_s = (F_s, X_s)$, and let $G \in \bigcap_s \mcal{P}_s$. In particular, $G \in \mcal{P}$ and by Lemma~\ref{lem:uemprop2-preserve},
$(T_0,T_1)$ is $C\oplus G$-\iatfshyp.
This completes the proof of Theorem~\ref{uemprop2}.
\end{proof}

The rest of this section will be dedicated to the proof of Theorem \ref{uemth0}, from which Theorem \ref{uemth1} follows.

\begin{theorem}\label{uemth0}
If a pair of $\msf{EM}$ instance
$(T_0,T_1)$ is $C$-\iatfshyp, then for every
$C$-computable $\msf{EM}$ instance $T$, there exists
a  solution $G$ to $T$ such that
either $C\oplus G$ does not compute a solution to $T_0$,
or $C\oplus G$ does not compute a solution to $T_1$.
\end{theorem}
\begin{proof}
For notational convenience,  we assume
$C=\emptyset$.
Fix $(T_0,T_1)$ and $T$ as in Theorem \ref{uemth0}.
By Proposition \ref{uemprop0}, we may assume that
$T$ is stable (for every $x\in\omega$, either $x\rightarrow_T y$ for all but finitely many $y$,
or $y\r_T x$ for all but finitely many $y$).

In the rest of the proof, every Turing functional $\Phi^G$ is computing a set of integers,
namely $\{n: \Phi^G(n)\downarrow =1\}$;
so it makes sense to write $\Phi^G\cap A$.
Let $A_0\sqcup A_1$ be a 2-partition (of $\omega$) such that
 $x\in A_0$ if and only if $x\r_T y$ for all but finitely many
$y\in \omega$ (which is well defined since $T$ is stable).
This automatically ensures that $x\in A_1$ iff $y\r_T x$ for all but finitely many
$y\in\omega$.
For a set $Z\subseteq \omega$, a $2$-partition
$X_0\sqcup X_1$ of $\omega$, we say $Z$ is \emph{\compatible} with $X_0\sqcup X_1$ if
$Z\cap X_0\r_T Z\cap X_1$.
Note that if $Z\subseteq X_i$ for some $i$, then $Z$ is \compatible\ with $X_0\sqcup X_1$.

A \emph{condition} is a Mathias condition $(F,X)$ with the following properties:
\begin{itemize}
\item [(a)] $F$ is  $T$-transitive and \compatible\ with $A_0\sqcup A_1$;
\item [(b)]
$F\cap A_0\r_T X\r_T F\cap A_1$;

\item [(c)] $(T_0,T_1)$ is $X$-\iatfshyp.

\end{itemize}
A \emph{precondition} is a Mathias condition satisfying (a)(c).

\begin{lemma}\label{uemlem3}
Let  $(F,X)$ be a condition and $Y\subseteq X$.
\begin{enumerate}

\item If $Y$ is  $T$-transitive, then
 $F\cup Y$ is  $T$-transitive.

 \item If $Y$ is    \compatible\ with $A_0\sqcup A_1$, then
 $F\cup Y$ is  \compatible\ with $A_0\sqcup A_1$.

 \item For every precondition $(E,Y)$, there is a $b\in\omega$
 so that $(E,Y\cap (b,\infty))$ is a condition.
\end{enumerate}
\end{lemma}
\begin{proof}
Item (1)(2)
follows from property (b) of $(F,X)$.
Item (3) follows from definition of $A_0,A_1$.
\end{proof}

\begin{lemma}\label{uemlem4}
For every condition $(F, X)$ there exists an extension
$(E, Y)$ such that $|E| > |F|$.
\end{lemma}
\begin{proof}
Let $x\in X\setminus F$.
Clearly $\{x\}$ is $T$-transitive and \compatible\ with $A_0\sqcup A_1$.
By Lemma \ref{uemlem3} item (1)(2),
$(F\cup\{x\}, X)$ is a precondition.
\end{proof}

For $e\in\omega$, $i\in 2$,
let $\mcal{R}_e^i$ denote the
requirement:
$$\Phi_e^G\text{ is not a solution to }T_i.$$
We will construct a solution $G$ of $T$ satisfying:
$$\mcal{R}_{e_0,e_1}: \mcal{R}_{e_0}^0\vee \mcal{R}_{e_1}^1$$
for all $e_0,e_1\in\omega$.
A condition $(F,X)$ \emph{forces}
$\mcal{R}_{e_0,e_1}$ if: for every
  solution  $G$ of $T$ satisfying $(F,X)$,
$G$ satisfies $\mcal{R}_{e_0,e_1}$.
Note that the definition of forcing
is slightly different from that in section \ref{sect:fs-ts-sem},
that we restrict to $T$-transitive set.
This restriction cannot be applied in section \ref{sect:fs-ts-sem} since
 there we deal with arbitrary instance instead of computable instance.

\begin{lemma}\label{uemlemmain}
For every condition $c$
and indices $e_0,e_1\in\omega$,
there is an extension of $c$ forcing $\mcal{R}_{e_0,e_1}$.
\end{lemma}
\begin{proof}
Fix $c=(D,X)$.
By Lemma \ref{uemlem3},
for notational convenience, we assume $D=\emptyset$ and $X=\omega$.
We firstly describe a process to partially compute a
$4$-\iatfs. Then we show that if the computation
diverges, we obtain an extension $d\leq c $ forcing $\mcal{R}_{e_0,e_1}$
in the $\Pi_1^0$ way: one of $\Phi_{e_i}^{G}$ is not infinite; and if the computation converges,
we obtain $d\leq c$ forcing $\mcal{R}_{e_0,e_1}$ in a
$\Sigma_1^0$ way: for some $n,m,l$ so that
$T_0,T_1$ \diagonalagainst\ $(V_n,V_{n,m}), (U_{n,m},U_{n,m,l})$ respectively,
either $\Phi_{e_0}^G\cap U_n\ne\emptyset\wedge\Phi_{e_0}^G\cap U_{n,m}\ne\emptyset$;
or $\Phi_{e_1}^G\cap V_{n,m}\ne\emptyset\wedge\Phi_{e_1}^G\cap V_{n,m,l}\ne\emptyset$
(which means for some $i<2$,  $\Phi_{e_i}^G$ is not a solution to $T_i$).

\bigskip

\emph{Defining $U_n$}.
Given $n \in \omega$, search computably for some
finite set $U_n>n$ such that
for every $8$-partition
  $X_0\sqcup\cdots\sqcup X_7 = \omega$, there exists an $i< 8$,
  a finite $T$-transitive set $E\subseteq X_i$
  such that  $\Phi_{e_0}^E\cap U_n\ne\emptyset $.
\smallskip

 \emph{Defining $U_{n,m},V_{n,m}$}.
 Given $n,m\in\omega$, search computably for some
 finite set $U_{n,m}>m, V_{n,m}>n$
such that
for every $4$-partition $X_0\sqcup\cdots\sqcup X_3$:
\begin{itemize}
\item [(a)]  either there exists an $i< 4$, a finite $T$-transitive set $E\subseteq X_i$
   such that   $\Phi_{e_0}^{E}\cap U_n \neq \emptyset$ and $\Phi_{e_0}^E\cap U_{n,m}\ne\emptyset$;
  \item [(b)] or there exist $j\ne i<4$ and two finite $T$-transitive
  sets $F\subseteq X_j, E\subseteq X_i$ such that
  $\Phi_{e_0}^E\cap U_n \neq \emptyset$ and $\Phi_{e_1}^{F}\cap V_{n,m}\ne\emptyset$.
  \end{itemize}

 \emph{Defining $U_{n,m,l},V_{n,m,l}$}.
  Given $n,m\in\omega$, search computably for some
 finite sets $U_{n,m,l}>m, V_{n,m,l}>l$
 such that for every $2$-partition $X_0\sqcup X_1$, there exists a   finite $T$-transitive set $E $ compatible with $X_0\sqcup X_1$
  such that:
  \begin{itemize}

  \item [(p)] either   $\Phi_{e_0}^E\cap U_n \neq \emptyset$ and $\Phi_{e_0}^E\cap U_{n,m,l}\ne\emptyset$;

  \item [(q)] or
  $\Phi_{e_1}^E\cap V_{n,m} \neq \emptyset$ and $\Phi_{e_1}^E\cap V_{n,m,l}\ne\emptyset$.
  \end{itemize}
\bigskip

\noindent\textbf{Case 1}: $U_n$ is not found for $n$.

This is straightforward.
By compactness, the following $\Pi_1^0$ class $\mcal{P}$
of $8$-partitions $X_0\sqcup\cdots\sqcup X_7$
is nonempty: for every $i< 8$, every $T$-transitive finite set $E\subseteq
X_i$, $\Phi_{e_0}^E\cap (n,\infty)=\emptyset$.
As $\msf{WKL}$ preserves \propiatfshyp\ (Proposition \ref{uemprop2}),
there exists a  member
$X_0\sqcup\cdots\sqcup X_7$ of $\mcal{P}$
so that $(T_0,T_1)$ is $\oplus_{i<8} X_i$-\iatfshyp.
Fix any $i < 8$ such that $X_i$ is infinite.
Then $(D,X_i) $ is an extension of $c$ forcing
$\mcal{R}_{e_0,e_1}$.

\ \\

\noindent\textbf{Case 2}: $U_n$ is found but not
$(U_{n,m},V_{n,m})$ for some $n,m$.

By compactness, the following $\Pi_1^0$ class $\mcal{P}$
of $4$-partitions $X_0\sqcup\cdots\sqcup X_3$
is nonempty:
\begin{itemize}
\item[(a)]  for every $i< 4$, every finite $T$-transitive set $E\subseteq X_i$,
   we have   $\Phi_{e_0}^{E}\cap U_n=\emptyset\vee\Phi_{e_0}^E\cap (  m,\infty)=\emptyset$; and
  \item[(b)] for every $j\ne i<4$ and every two finite $T$-transitive
  sets $F\subseteq X_j, E\subseteq X_i$, we have
  $\Phi_{e_0}^E\cap U_n=\emptyset\vee\Phi_{e_1}^{F}\cap (n,\infty)=\emptyset$.
  \end{itemize}

As $\msf{WKL}$ preserves \propiatfshyp\ (Proposition \ref{uemprop2}),
there exists a  member $X_0\sqcup\cdots\sqcup X_3$ of $\mcal{P}$
so that
$(T_0,T_1)$ is $\oplus_{i<4} X_i$-\iatfshyp.
Consider the
 $8$-partition $(X_i\cap A_k: i<4,k<2)$. By definition of
$U_n$, there exist $i<4,k<2$ and a finite $T$-transitive set
 $E\subseteq X_i\cap A_k$,
 such that
$\Phi_{e_0}^E\cap U_n\ne\emptyset$.

\noindent\textbf{Subcase 1}: $X_i$ is infinite.

Since $E\subseteq A_k$,
$E$ is compatible with $A_0\sqcup A_1$.
So
$d=(E,X_i )$ is a  precondition extending $c$.
Note that by property (a) of $X_0\sqcup\cdots\sqcup X_3$,
for every $T$-transitive set $G$ satisfying $d$ (so $G\subseteq X_i$), we have
$\Phi_{e_0}^G\cap (  m,\infty)=\emptyset$ (since $\Phi_{e_0}^G\cap U_n\ne\emptyset$).
Thus, $d$ forces
$\mcal{R}_{e_0,e_1}$.

\noindent\textbf{Subcase 2}:
$X_i$ is finite.

Then there exists a $j\ne i $ such that $X_j$ is infinite.
Note that  by property (b) of   $X_0\sqcup\cdots\sqcup X_3$,
for every $T$-transitive set $G\subseteq X_j$,
we have
$\Phi_{e_1}^G\cap (n,\infty)=\emptyset$.
Thus the condition $(D,X_j)$ extends $c$ and forces
$\mcal{R}_{e_0,e_1}$.

 \ \\

\noindent\textbf{Case 3}: $U_n,U_{n,m},V_{n,m}$ are found but not
$(U_{n,m,l},V_{n,m,l})$
for some $n,m,l\in\omega$.

By compactness, the following $\Pi_1^0$ class $\mcal{P}$
of $2$-partitions $X_0\sqcup  X_1$
is nonempty: for every   finite $T$-transitive finite set $E $ compatible with $X_0\sqcup X_1$,
  we have
 \begin{itemize}

  \item[(p)]  $\Phi_{e_0}^E\cap U_n=\emptyset\vee \Phi_{e_0}^E\cap(  m,\infty)=\emptyset$; and

  \item[(q)]
  $\Phi_{e_1}^E\cap V_{n,m}=\emptyset\vee \Phi_{e_1}^E\cap (  l,\infty)=\emptyset$.
  \end{itemize}

By Proposition \ref{uemprop2},
there exists a member  $X_0\sqcup X_1$ of $\mcal{P}$
so that $(T_0,T_1)$ is $\oplus_{j<2} X_j$-\iatfshyp.
Consider the $4$-partition $(X_j\cap A_k:j,k<2)$.
By property (p) of $X_0\sqcup X_1$,
for every $j,k<2$ and $E\subseteq X_j\cap A_k$
(so $E$ is \compatible\ with $X_0\sqcup X_1$),
\begin{align}\label{uemeq0}
\Phi_{e_0}^E\cap U_n=\emptyset\vee \Phi_{e_0}^E\cap(  m,\infty)=\emptyset.
\end{align}
Combine with
the  definition of $U_{n,m},V_{n,m}$ (where we take the 4-partition to be $(X_j\cap A_k:j,k<2)$
and note that by (\ref{uemeq0}), property (a)   fails, so property (b) occurs),
we have: there are $(j,k)\ne (\h j,\h k)$ and finite $T$-transitive sets
$E\subseteq X_j\cap A_k, F\subseteq X_{\h j}\cap A_{\h k} $
such that
$$\Phi_{e_0}^E\cap U_n,
\Phi_{e_1}^{F}\cap V_{n,m}\ne\emptyset.$$

\noindent\textbf{Subcase 1}:
Either $X_j$ or $X_{\h j}$ is infinite.

Suppose $X_j$ is infinite.
Consider the precondition $d=(E,X_j)$ extending $c$.
Note that for every $T$-transitive set $G $ satisfying $d$,
$ G$ is \compatible\ with $X_0\sqcup X_1$ (since $G\subseteq X_j$).
 Thus by property (p) of $X_0\sqcup X_1$   and $\Phi_{e_0}^E\cap U_n\ne\emptyset$,
for every  $T$-transitive set $G $  satisfying $d$,
we have $\Phi_{e_0}^G\cap (m,\infty)=\emptyset$.
Thus $d$
forces
$\mcal{R}_{e_0,e_1}$. Suppose now $X_{\h j}$ is infinite. Then, taking $d = (F, X_{\h j})$, a similar argument shows that by property (q) of $X_0\sqcup X_1$, for every  $T$-transitive set $G $  satisfying $d$,
we have $\Phi_{e_1}^G\cap (l,\infty)=\emptyset$.
Thus $d$
forces
$\mcal{R}_{e_0,e_1}$. Thus we are done in this subcase.
\bigskip

\noindent\textbf{Subcase 2}:
Both  $X_j$, $X_{\h j}$ are finite.

This implies $j=\h j$ since $X_0 \sqcup X_1 = \omega$, and then $k\ne \h k$ since $(j,k)\ne (\h j,\h k)$.
Let $b$ be sufficiently large to witness the limits of the elements of $F$ and $E$ with respect to the tournament.

If $j=\h j=0$ (so $X_1$ is infinite) and $\h k=0, k=1$, consider the
condition $d= (F,X_1\setminus[0,b])$.
Since $F\subseteq A_{\h k}=A_0$,
 we have $F\r_T X_1\setminus [0,b]$.
 Therefore for every   $G\subseteq X_1\setminus [0,b]$,
 $F\cup G$ is \compatible\ with $X_0\sqcup X_1$ (since $F\subseteq X_0$).
 Thus for every $T$-transitive set $G$ satisfying $d$,
 by property (q) of $X_0\sqcup X_1$
 and since $\Phi_{e_1}^F\cap V_{n,m}\ne\emptyset$, we have $\Phi_{e_1}^{G}\cap (l,\infty)=\emptyset$.
Thus $d$ forces $\mcal{R}_{e_1}^1$.

If $j=\h j=1$ (so $X_0$ is infinite) and $\h k=0, k=1$, consider the condition $d=(E,X_0 \setminus[0,b])$.
Since $E\subseteq A_k=A_1 $,
we have $X_0\setminus[0,b]\r_T E$.
Therefore for every set $G\subseteq X_0\setminus[0,b]$,
$E\cup G$ is \compatible\ with $X_0\sqcup X_1$ (since $E\subseteq X_1$). Thus for every  $T$-transitive set $G$
satisfying $d$,
by property (p) of $X_0\sqcup X_1$ and $\Phi_{e_0}^E\cap U_n\ne\emptyset$, we have
$\Phi_{e_0}^G\cap (m,\infty)=\emptyset$.
Thus $d$ forces $\mcal{R}_{e_0}^0$.

If $j=\h j=0$ (so $X_1$ is infinite) and $\h k=1, k=0$,
then the condition $(E, X_1 \setminus [0, b])$ forces $\mcal{R}_{e_0}^0$ by a similar argument using property (p) of $X_0\sqcup X_1$.

If $j=\h j=1$ (so $X_0$ is infinite) and $\h k=1, k=0$,
then the condition $(F, X_0 \setminus [0, b])$ forces $\mcal{R}_{e_1}^1$ by a similar argument using property (q) of $X_0\sqcup X_1$.

 \ \\

\noindent
\textbf{Case 4}:  $U_n,U_{n,m},V_{n,m},U_{n,m,l},V_{n,m,l}$
  is found for all $n,m,l$.

Since $(T_0,T_1)$ is \iatfshyp, there exist
$n, m,l$ such that $T_0$ and $T_1$ diagonalize against\
$(U_{n},U_{n,m,l})$ and
$(V_{n, m},V_{n, m,l})$, respectively.
By definition of
$U_{n,m,l},V_{n,m,l}$ (where we take $X_0\sqcup X_1$ to be $A_0\sqcup A_1$),
there exists a finite $T$-transitive set $F $   compatible  with
$A_0\sqcup A_1$  such that
\begin{itemize}
    \item[(p)]  either
$\Phi_{e_0}^{F}\cap U_{n} \neq \emptyset$ and $\Phi_{e_0}^{F}\cap U_{n,m,l}\ne\emptyset$;
    \item[(q)] or
$\Phi_{e_1}^{F}\cap V_{n,m} \neq \emptyset$ and $\Phi_{e_1}^{F}\cap V_{n,m,l}\ne\emptyset$.
\end{itemize}
Let $d=(F,X)$. We claim that $d$ forces $\mcal{R}_{e_0,e_1}$ by forcing $\mcal{R}^0_{e_0}$ on case (p) and $\mcal{R}^1_{e_1}$ on case (q).
Let $G$ be a set satisfying $d$.
In the case (p), $\Phi_{e_0}^G$ has a non-empty intersection with both $U_{n}$ and $U_{n,m,l}$, in which case $\Phi_{e_0}^G$ is not a solution to $T_0$ (recall the definition of diagonalizes against); and in the case (q), $\Phi_{e_1}^G$ has non-empty intersection with both $V_{n,m}$ and $V_{n,m,l}$, in which case $\Phi_{e_1}^G$ is not a solution to $T_1$.
This completes the proof of the lemma.
\end{proof}

 Let $\Fcal = \{c_0, c_1, \dots \}$ be a sufficiently generic filter for this notion of forcing,
where $c_s = (F_s, X_s)$, and let $G = \bigcup_s F_s$.
By property (a) of a condition, $G$ is $T$-transitive.
By Lemma~\ref{uemlem4}, $G$ is infinite,
and by Lemma~\ref{uemlemmain},
$G$ satisfies $\mcal{R}_{e_0,e_1}$ for all $e_0,e_1$.
 By pairing
argument, this means either $G$ does not compute
a solution to $T_0$, or $G $ does not compute a solution to $T_1$.
This completes the proof of Theorem~\ref{uemth0}.
\end{proof} 

\section*{Acknowledgements}

The second author was partially supported by grant ANR ``ACTC'' \#ANR-19-CE48-0012-01.

\bibliographystyle{plain}
\bibliography{bibliography}

\appendix


\end{document}